\documentclass[a4paper,oneside,11pt]{amsart}

\reversemarginpar

\usepackage[colorlinks,hyperindex,linkcolor=blue,urlcolor=black,pdftitle={On isometric asymptotes for operators quasisimilar to isometries},pdfauthor={Maria F.Gamal'}]
{hyperref}
\usepackage{amsmath}
\usepackage{amsfonts}
\usepackage{amssymb}
\usepackage{amsthm}

\usepackage[english]{babel}
\usepackage{enumerate}

\usepackage{graphicx}
\usepackage{color}

\usepackage[abbrev]{amsrefs}

\numberwithin{equation}{section}

\theoremstyle{plain}
\newtheorem{theorem}{Theorem}[section]
\newtheorem{lemma}[theorem]{Lemma}
\newtheorem{corollary}[theorem]{Corollary}
\newtheorem{proposition}[theorem]{Proposition}
\newtheorem{theoremcite}{Theorem}

\theoremstyle{definition}
\newtheorem{remark}[theorem]{Remark}
\newtheorem{example}[theorem]{Example}

\begin{document}

\title[On isometric asymptotes]{On isometric asymptotes of operators quasisimilar to isometries}
\author{Maria F. Gamal'}
\address{
 St. Petersburg Branch\\ V. A. Steklov Institute of Mathematics\\
Russian Academy of Sciences\\ Fontanka 27, St. Petersburg\\ 
191023, Russia  }
\email{gamal@pdmi.ras.ru}

%\thanks{The second author has been supported by a grant of the
\maketitle

\begin{abstract}
The notion of isometric and unitary asymptotes was introduced for power bounded operators in 1989  and was  generalized 
% on operators which are not necessarily power bounded 
in 2016--2019 by  K\'erchy. In particular, it was shown that there exist operators without unitary asymptote. In this paper  operators are constructed which are quasisimilar to  isometries and do not have isometric asymptotes.   
Also a contraction is constructed which is quasisimilar to the unilateral shift of infinite multiplicity and whose  isometric asymptote  contains a (non-zero) unitary summand.


%\date\today
\end{abstract}

%\subjclass%[2010]
%[2020]{47B02, 47A45, 47A99}

%\keywords{unilateral shift, isometry, unitary operator, quasisimilarity, unitary asymptote}
%\date\today

%FORMULA \mathop{\text{\rm ess\,sup}}

\section{Introduction}

Let $\mathcal H$ be a (complex, separable) Hilbert space, and let 
 $T$ be a   (linear, bounded) operator on $\mathcal H$. The spectrum and the point spectrum of $T$ are denoted by
 $\sigma(T)$ and $\sigma_p(T)$, respectively. 
The \emph{(cyclic) multiplicity} $\mu_T$
 of an operator $T$ 
is the minimal dimension $N$, $0\leq N\leq \infty$, of  subspaces $\mathcal E\subset\mathcal H$ such that 
$\mathcal H=\vee_{k=0}^\infty T^k\mathcal E$. An operator $T$ 
is called \emph{cyclic}, if $\mu_T=1$. 
The lattice of all (closed) subspaces $\mathcal E$ of $\mathcal H$ such that $T\mathcal E\subset\mathcal E$ is called 
the \emph{invariant subspace lattice} of $T$ and is denoted by $\operatorname{Lat}T$. 
The lattice of all  (closed) subspaces $\mathcal E$ of $\mathcal H$ such that $A\mathcal E\subset\mathcal E$ 
for every operator $A$ such that $AT=TA$
is called the \emph{hyperinvariant subspace lattice} of $T$ and is denoted by $\operatorname{Hlat}T$. 

Let $\mathcal H$  and $\mathcal K$ be two Hilbert spaces. Denote by $\mathcal L(\mathcal H, \mathcal K)$ 
the space of all  (linear, bounded) transformations acting from $\mathcal H$ to  $\mathcal K$. 
Set $\mathcal L(\mathcal H)=\mathcal L(\mathcal H, \mathcal H)$, then $\mathcal L(\mathcal H)$ is 
 the algebra of all (linear, bounded) operators acting on $\mathcal H$. 
Let $T\in\mathcal L(\mathcal H)$,
$R\in\mathcal L(\mathcal K)$, $X\in\mathcal L(\mathcal H, \mathcal K)$ be such that $X$ \emph{intertwines} $T$ with $R$, that is, $XT=RX$. 
If $\operatorname{clos}X\mathcal H=\mathcal K$, then we write $T\buildrel d \over\prec R$. A transformation 
 $X\in\mathcal L(\mathcal H, \mathcal K)$ is called a \emph{quasiaffinity}, if $\ker X=\{0\}$ and 
$\operatorname{clos}X\mathcal H=\mathcal K$. 
If $X$ is a quasiaffinity, then $T$ is called a \emph{quasiaffine transform} of $R$, in notation: $T\prec R$. 
If $T\prec R$ and $R\prec T$, then $T$ and $R$ are called \emph{quasisimilar}, in notation: $T\sim R$. 
If $X$ is \emph{invertible}, that is, 
$X^{-1}\in\mathcal L(\mathcal K, \mathcal H)$, then $T$ and $R$ are called \emph{similar}, in notation: $T\approx R$.  If, in addition, 
$X$ is a unitary transformation, then $T$ and $R$ are called \emph{unitarily equivalent}, in notation: $T\cong R$. 
Recall that if $U$ and $U_1$ are two unitary operators such that $U\prec U_1$ or 
$U\buildrel d \over\prec U_1\buildrel d \over\prec U$,  then $U\cong U_1$, 
and if $V$ and $V_1$ are two isometries such that $V\sim V_1$, then $V\cong V_1$ ({\cite[Proposition II.3.4]{nfbk}}, 
{\cite[Lemma II.13.6, Propositions II.10.6 and II.13.7]{conway}}).

An operator  $T\in\mathcal L(\mathcal H)$ is called \emph{power bounded}, if $\sup_{n\in\mathbb N}\|T^n\|<\infty$.
 A power bounded operator $T$ is 
\emph{of class} $C_{1\cdot}$, if $\inf_{n\in\mathbb N}\|T^n x\|>0$ for every $0\neq x\in\mathcal H$, 
and is \emph{of class} $C_{0\cdot}$, if $\lim_n\|T^n x\|=0$ for every $ x\in\mathcal H$. 
$T$ is \emph{of class} $C_{\cdot a}$, if $T^*$ is of class $C_{a\cdot}$, and $T$ is \emph{of class} $C_{ab}$, 
if $T$ is of classes $C_{a\cdot}$ and $C_{\cdot b}$, $a,b=0,1$.
 
An operator $T$ is called a {\it contraction} if $\|T\|\leq 1$. Clearly, a contraction is power bounded. 

The notion of isometric and unitary asymptotes was introduced in \cite{ker89} for power bounded operators. In particular, it was proved 
in \cite{ker89} that every power bounded operator has an isometric asymptote constructed using a Banach limit.
 In \cite{ker16} and  \cite{kerntuple} (see also {\cite[Sec. IX.1]{nfbk}}) the notion of unitary asymptote was  generalized  on operators which are not necessarily power bounded. 
Recall the definitions.  
A pair $(X,U)$, where $U\in\mathcal L(\mathcal K)$ is an isometry (a unitary operator), and $X$ is a 
 transformation such that $XT=UX$, 
 is called an \emph{isometric} (a \emph{unitary}) \emph{asymptote} of an operator $T\in\mathcal L(\mathcal H)$, if for any other pair $(Y,V)$, 
where $V$ is an isometry (a unitary operator), and $Y$ is a  transformation such that $YT=VY$, there exists 
a unique transformation $Z$ such that $ZU=VZ$ and $Y=ZX$. The uniqueness of $Z$ is equivalent to the relation 
$\operatorname{clos}X\mathcal H=\mathcal K$ for an isometric asymptote, and to the relation  
$\vee_{n\geq 0}U^{-n}X\mathcal H = \mathcal K$ for a unitary asymptote (that is, the pair $(X,U)$ is minimal). 
The transformation $X$ is called the \emph{canonical intertwining mapping}, and the isometry $U$ sometime also 
\emph{will be called the isometric (unitary) asymptote}. 

It follows from the definition of an isometric asymptote that if $T$ is an operator and $V$ is an isometry such that 
  $T\approx V$, then $V$ is an isometric asymptote of $T$. 

If $(X,V)$ is an isometric asymptote of $T\in\mathcal L(\mathcal H)$, then the pair $(Y,U)$, where $U$ is the minimal unitary extension of $V$ 
and $Y$ is a natural extension of $X$, is a unitary asymptote of $T$. In particular, if the isometry $V$ is unitary, then the 
isometric and unitary asymptotes of $T$ are equal.  Conversely, if  $(Y,U)$ is a unitary asymptote of $T$, then 
$(X,U|_{\operatorname{clos}Y\mathcal H})$, where $X$ is equal to $Y$ considered as a transformation from $\mathcal H$ 
into $\operatorname{clos}Y\mathcal H$, is an isometric asymptote of $T$. Thus, \emph{$T$ has an isometric asymptote 
 if and only if $T$ has a unitary asymptote}. (See \cite{ker89} and  {\cite[Sec. IX.1]{nfbk}}.)

If $T$ has an isometric asymptote $(X,V)$ and $T$ is invertible, then $V$ is unitary and $(X,V^{-1})$ 
is an isometric (and a unitary) asymptote of $T^{-1}$ {\cite[Proposition 22(a)]{kerntuple}}. 

Two pairs $(X,U)$ and $(X_1,U_1)$,
 where $U$ and $U_1$ are isometric (unitary) operators, and $X$ and $X_1$ are  transformations such that $XT=UX$ and $X_1T=U_1X_1$, 
are  \emph{similar}, if there exists an \emph{invertible} transformation $Z$
such that $ZU=U_1Z$ and $X_1=ZX$. It follows from the definition  that 
 an isometric (a unitary) asymptote of $T$ is defined up to similarity. 
If two operators are similar and one of them has an isometric (a unitary) asymptote, then the other has an isometric (a unitary) 
 asymptote, too, 
and their isometric (unitary) asymptotes are similar \cite{ker89}. 

As was mentioned above, for power bounded operators an isometric asymptote can be constructed using a Banach limit \cite{ker89}. Moreover, if  $T\in\mathcal L(\mathcal H)$ is a contraction, then there exists a canonical intertwining mapping $X$ for $T$ such that 
\begin{equation}\label{lim} \lim_{n\to\infty}(T^nx_1, T^nx_2)=(Xx_1,Xx_2) 
\text{ for all } x_1,x_2\in\mathcal H \end{equation}
{\cite[Sec. IX.1]{nfbk}}. 

The following notation will be used. For a Hilbert space $\mathcal H$ and its subspace $\mathcal M$ by $I_{\mathcal H}$ and $P_{\mathcal M}$ the identical operator on $\mathcal H$ and the orthogonal projection onto $\mathcal M$ are denoted, respectively.

$\mathbb D$ is the open unit disc, $\overline{\mathbb D}$ is the closed unit disc,  $\mathbb T$ is the unit circle, 
$m$ is the normalized Lebesgue (arc length) measure on $\mathbb T$. A unitary operator is called absolutely continuous, if  its spectral measure is absolutely continuous with respect to $m$. 

$H^p$ is the Hardy space on $\mathbb D$ ($p=2, \infty$), 
$S\in \mathcal L(H^2)$ is the operator of multiplication by the independent variable on $H^2$ (the  \emph{unilateral shift}),
$U_{\mathbb T}\in \mathcal L(L^2(\mathbb T,m))$ is the operator of multiplication by the independent variable on $L^2(\mathbb T,m)$
 (the  \emph{bilateral shift}). 
For a cardinal number $0\leq n\leq \infty$ the \emph{unilateral shift of multiplicity} $n$ is 
$S_n=\oplus_{k=1}^n S$ which acts on $H^2_n=\oplus_{k=1}^n H^2$. For $n=0$, $S_n$ is the zero operator acting on the zero space $\{0\}$, 
for $n=1$ we have  $S_1=S$ and $H^2_1=H^2$. It is well known and easy to see that $\mu_{S_n}=n$, $0\leq n\leq \infty$.

If $T$ is a power bounded operator (on a separable Hilbert space), then $\sigma(T)\subset\overline{\mathbb D}$,  and 
$\sigma_p(T)\cap\mathbb T$ is at most countable by \cite{jamison}.
The relationship between the peripheral point spectrum 
$\sigma_p(T)\cap\mathbb T$ for operators $T$ with 
$\sigma(T)\subset\overline{\mathbb D}$ and growth
of  $\|T^n\|$  is considered in \cite{elfallahransford}, \cite{eisnergrivaux}, see also references therein. 
In  particular, examples of operators with uncountable  peripheral point spectrum  was given there. 
In \cite{ma1}, \cite{ma2} one-dimentional perturbations $T$ of unitaries are constructed such that $\sigma(T)\subset \mathbb T$ and $\sigma_p(T)$ is  uncountable.

The proof of the following theorem  is contained in {\cite[Example 14]{kerntuple}}. 

\begin{theorem}\label{thmeigen}\cite{kerntuple}. Let $T$ be an operator (on a separable Hilbert space). If $\sigma_p(T^*)\cap\mathbb T$ is uncountable, 
then $T$ does not have an isometric asymptote. 
\end{theorem}

\begin{proof} For  $\zeta\in \sigma_p(T^*)$ set $\mathcal E_\zeta=\ker(T^*-\zeta I)$. 
Then $P_{\mathcal E_\zeta} T= \overline\zeta I_{\mathcal E_\zeta}P_{\mathcal E_\zeta}$. If $T$ has an isometric asymptote $(X,V)$, 
then for every $\zeta\in \sigma_p(T^*)\cap\mathbb T$ there exists a transformation $Z_\zeta$ such that 
 $P_{\mathcal E_\zeta}=Z_\zeta X$ and $Z_\zeta V=\overline\zeta I_{\mathcal E_\zeta}Z_\zeta$. Since the range of
 $Z_\zeta$ is dense, 
we have $Z_\zeta^*$ is injective. Therefore, $\zeta\in\sigma_p(V^*)$. Thus, $\sigma_p(T^*)\cap\mathbb T\subset \sigma_p(V^*)$. 
Since 
 $\sigma_p(V^*)\cap\mathbb T$ is at most countable, 
$\sigma_p(T^*)\cap\mathbb T$ 
must be at most countable, a contadiction.
\end{proof}

In Proposition \ref{thmweight} below an example is given which seems to be more natural than 
 {\cite[Example 14]{kerntuple}}. In particular, the operators $ \mathcal S_\omega$ and $\mathcal S_\omega^{-1}$ from Proposition \ref{thmweight} have  non-zero unitary asymptotes. Also, the weight $\omega$ can be chosen such that
 $\sigma( \mathcal S_\omega)=\sigma( \mathcal S_\omega^{-1})=\mathbb T$, see  {\cite[Proposition 2.3]{est}} 

Consider a weight $\omega\colon\mathbb Z \to [1,\infty)$ with the following properties: $\omega$ is unbounded,   nonincreasing,  and  $\omega(n)=1$ for all $n\geq 0$. 
Set $$L^2_\omega(\mathbb T)=\Bigl\{f\in L^2(\mathbb T,m)\ :\ \|f\|^2_\omega=\sum_{n\in\mathbb Z}|\widehat f(n)|^2\omega^2(n)<\infty\Bigr\}.$$
The \emph{bilateral weighted shift}  $\mathcal S_\omega\in\mathcal L(L^2_\omega(\mathbb T))$ is the operator 
of multiplication by the independent variable on $L^2_\omega(\mathbb T)$. 
It is known that $ \mathcal S_\omega$ is a contraction of class $C_{10}$, and $ \mathcal S_\omega$ is invertible if and only if 
\begin{equation}\label{omegainv}\sup_{n\in\mathbb Z}\frac{\omega(n-1)}{\omega(n)}<\infty. \end{equation}
For references, see   {\cite[Sec. 2]{est}} and {\cite[Sec. 7]{kersz}}. 

Denote by $Y_\omega$ the natural imbedding of $L^2_\omega(\mathbb T)$ into $L^2(\mathbb T,m)$. 
 It is easy to see that  $Y_\omega$
is a quasiaffinity, and   
$\lim_k(\mathcal S_\omega^kf,\mathcal S_\omega^kg)_{L^2_\omega(\mathbb T)}
=(Y_\omega f, Y_\omega g)_{ L^2(\mathbb T,m)}$ for every $f$, $g\in L^2_\omega(\mathbb T)$. Therefore,   
   $(Y_\omega, U_{\mathbb T})$ is an isometric (and unitary) asymptote of $\mathcal S_\omega$ satisfying \eqref{lim}.

The following proposition is {\cite[Theorem 3.11(2)]{nik}} formulated in the form convenient for our purpose. 

\begin{proposition}\cite{nik}\label{thmweight}
 Suppose that $\omega\colon\mathbb Z \to [1,\infty)$ is unbounded,   nonincreasing,  $\omega(n)=1$ for all $n\geq 0$,  
$\omega$ satisfies \eqref{omegainv}, and $$\sum_{n=1}^\infty\frac{1}{\omega^2(-n)}<\infty.$$ 
 Set $\mathcal M=\{f\in L^2_\omega(\mathbb T)\ :\ \widehat f(n)=0 \text{ for } n\geq 1\}$.
Then $\mathcal M\in\operatorname{Lat}\mathcal S_\omega^{-1}$ and
 $\overline{\mathbb D}\subset\sigma_p((\mathcal S_\omega^{-1}|_{\mathcal M})^*)$.
Furthermore, $\mathcal S_\omega^{-1}$ has an isometric asymptote $(Y_\omega, U_{\mathbb T}^{-1})$ and $\mathcal S_\omega^{-1}|_{\mathcal M}$ does not have an  isometric asymptote. 
\end{proposition}

\begin{proof} It is easy to see  that $f\in\mathcal M$ is an eigenvector of $(\mathcal S_\omega^{-1}|_{\mathcal M})^*$ for an eigenvalue 
$\zeta\neq 0$ if and only if 
$\widehat f(-n)=\zeta^n \frac{\widehat f(0)}{\omega^2(-n)}$ for $n\geq 0$. For $\zeta=0$, an eigenvector is $f\in\mathcal M$ such that $\widehat f(0)\neq 0$ and $\widehat f(n)=0$ for $n\leq -1$. Thus, the claim about 
$\sigma_p((\mathcal S_\omega^{-1}|_{\mathcal M})^*)$ is proved. It remains to apply  
   {\cite[Proposition 22(a)]{kerntuple}} and Theorem \ref{thmeigen}. 
\end{proof}

The following simple lemma is well known and is added for the sake of completeness. 

\begin{lemma}\label{lemd} Suppose that $0\leq n<\infty$, $U$ is a unitary operator, $V$ is an isometry, and 
$U\oplus S_n\buildrel d \over\prec   V\buildrel d \over\prec  U\oplus S_n$. 
Then $V\cong U\oplus S_n$. 
\end{lemma}

\begin{proof} 
  Without loss of generality we can assume that $V=U_1\oplus S_k$, where $U_1$ is unitary and $0\leq k\leq\infty$.
It is easy to see that  if $R$ and $R_1$ are two operators such that $R\buildrel d \over\prec R_1$, then 
$\dim\ker R_1^*\leq \dim\ker R^*$.
Therefore  $k=n$. 
 
Denote by $\mathcal K$ and $\mathcal K_1$ the spaces on which $U$ and $U_1$ act, respectively, 
and by $Y$ a transformation which realizes the relation $U\oplus S_n\buildrel d \over\prec   V$.  
  Set $\mathcal M=\operatorname{clos}Y\mathcal K$. Since $U\buildrel d \over\prec V|_{\mathcal M}$, we have 
$V|_{\mathcal M}$ is of class  $C_{\cdot 1}$. Therefore, $\mathcal M\subset\mathcal  K_1$ and  
 $\mathcal M$ is a reducing invariant subspace for $U_1$. 
Furthermore, 
\begin{equation}\label{multn}S_n\buildrel d \over\prec U_1|_{\mathcal K_1\ominus \mathcal M}\oplus S_n.\end{equation} 
Relation \eqref{multn} implies that $U_1|_{\mathcal K_1\ominus \mathcal M}$ is an absolutely continuous unitary operator.  Therefore, 
 $$\mu_{U_1|_{\mathcal K_1\ominus \mathcal M}\oplus S_n}=\mu_{U_1|_{\mathcal K_1\ominus \mathcal M}}+n\leq n,$$
where the latest estimate follows from \eqref{multn}.  
Since $n<\infty$, we conclude that $\mu_{U_1|_{\mathcal K_1\ominus \mathcal M}}=0$. This means that $\mathcal M=\mathcal K_1$. Thus, $U\buildrel d \over\prec U_1$. Similarly, $U_1\buildrel d \over\prec U$.  By 
{\cite[Proposition II.3.4]{nfbk}} and {\cite[Proposition II.10.6 and Lemma II.13.6]{conway}}, $U\cong U_1$. 
\end{proof}

In the following simple lemmas we find  isometric asymptotes of operators intertwining with isometries 
under the assumption that they have isometric asymptotes.

\begin{lemma}\label{lemfinite} Suppose that $0\leq n<\infty$, $U$ is a unitary operator, and $T$ is an operator 
such that $U\oplus S_n\buildrel d \over\prec  T\buildrel d \over\prec  U\oplus S_n$. 
 If $T$ has an isometric asymptote $V$, then $V\cong U\oplus S_n$. 
\end{lemma}

\begin{proof} It follows from the definition of an isometric asymptote that 
$T\buildrel d \over\prec   V\buildrel d \over\prec  U\oplus S_n$. 
Therefore, $n$, $U$, and $V$ satisfy the assumptions of Lemma \ref{lemd}.
\end{proof}

\begin{lemma}\label{lemshift} Suppose that $1\leq n<\infty$, $R$ is a contraction, and $T$ is an operator such that $R\prec T\prec S_n$. 
If $T$ has an isometric asymptote $V$, then $V\cong  S_n$. 
\end{lemma}

\begin{proof}
 Since $T\prec S_n$, we have $T\prec V\buildrel d \over\prec S_n$. 
By {\cite[Lemma 2.1]{gam}} or {\cite[Proposition 9]{kersz}}, the isometric asymptote of $R$ is $S_n$. Since $R\prec V$, we have 
$S_n\buildrel d \over\prec V$. Thus, $n$ and $V$ satisfy the assumption of  Lemma \ref{lemd}.   
\end{proof}

\begin{remark}\label{remshift} In Lemma \ref{lemshift}, $S_n$ cannot be replaced by $U\oplus S_n$ for an absolutely continuous unitary operator $U$ 
due to the relation $S_k\prec U$ for every $k\geq\mu_U$. 
\end{remark}

Clearly, if $T$ satisfies the assumptions of Theorem \ref{thmeigen}, then $T$ cannot be quasisimilar to an isometry. 
One of the purposes of this paper is to construct operators $T$ such that $T$ are quasisimilar to  isometries and do not have 
isometric asymptotes. This is done in Sec. 2, 3, 4. Namely, in Sec. 2 some general statements are proved which are used in Sec. 3
to construct examples of operators quasisimilar to cyclic unitaries with pure atomic spectrum and do not have unitary asymptotes. 
In Sec. 4 examples of operators $T$ such that $T\sim S$ and $T$ does not have isometric asymptote are constructed.  

 Another  purpose of this paper is to construct a contraction $T$ such that $T\sim S_\infty$ and 
an isometric asymptote of $T$ contains a (non-zero) unitary summand.  In particular, this shows that Lemmas \ref{lemfinite} and \ref{lemshift} cannot be generalized
to $n=\infty$. This is done in Sec. 5. (Actually, an operator $T$ 
which is similar to a contraction $T_1$ is constructed, but if $T\approx T_1$, 
then the isometries from the isometric asymptotes of $T$ and $T_1$ are unitarily equivalent.) 
Sec. 4 and Sec. 5 are independent from Sec. 2 and Sec. 3 and from each other.

\section{Operators quasisisimilar to unitaries}

Operators considered in this section are very close to ones considered in \cite{apostol} and 
\cite{ker87} (see also references in \cite{ker87}).  The difference is that existence of isometric (unitary) asymptotes of such operators is considered here. 

Suppose that $\{\mathcal K_{0n}\}_{n=0}^\infty$ and $\{U_{0n}\}_{n=0}^\infty$ are families of Hilbert 
spaces and reductive unitary operators such that $U_{0n}\in\mathcal L(\mathcal K_{0n})$ for all $n\geq 0$, and 
\begin{equation}\label{sing}\text{the scalar spectral measures of }U_{0n}\text{  are pairwise singular}, \ n\geq 0. 
\end{equation}
Suppose that $\mathcal H$ is a Hilbert space,  $T\in  \mathcal L(\mathcal H)$, and there exists 
 $\{0\}\neq\mathcal H_n\in\operatorname{Lat}T$ and  $X_n\in\mathcal L(\mathcal H_n, \mathcal K_{0n})$ such that 
\begin{equation}\label{main0}\mathcal H=\mathcal H_n\dotplus\vee_{k\neq n}\mathcal H_k\  \text{ for every  } n\geq 0,  \end{equation}
\begin{equation}\label{main1} X_n \text{  is invertible, and } X_nT|_{\mathcal H_n}=U_{0n}X_n\  \text{ for every  } n\geq 0. \end{equation}
Denote by $\mathcal Q_n$ the skew projection onto $\mathcal H_n$ parallel to $ \vee_{k\neq n}\mathcal H_k$. 
 Set $\mathcal K_0=\oplus_{n=0}^\infty\mathcal K_{0n}$ and $U_0=\oplus_{n=0}^\infty U_{0n}$. 

\begin{lemma}\label{lem1} Suppose that $\mathcal H$, $\mathcal K$ are Hilbert spaces,  $U\in\mathcal L(\mathcal K)$ is an isometry,  
  $T\in  \mathcal L(\mathcal H)$ satisfies \eqref{sing}, \eqref{main0} and \eqref{main1}, $X\in\mathcal L(\mathcal H,\mathcal K)$ 
is such that $XT=UX$. For $n\geq 0$, set $\mathcal K_n=\operatorname{clos}X\mathcal H_n$ and $U_n=U|_{\mathcal K_n}$. Then 
the following relations are fulfilled.
 \begin{enumerate}[\upshape (i)]
\item For every $n\geq 0$, $U_n$ is unitary, and there exists  
$\mathcal K_{1n}\in\operatorname{Lat}U_{0n}$ such that $U_n\cong U_{0n}|_{\mathcal K_{1n}}$.
\item The subspaces $\mathcal K_n$, $n\geq 0$, are pairwise orthogonal. 
\item $\operatorname{clos}X\mathcal H=\oplus_{n=0}^\infty \mathcal K_n$. 
\item $ \|X|_{\mathcal H_n}\mathcal Q_n\|\leq\|X\|$ for all $n\geq 0$.
\end{enumerate}
\end{lemma}

\begin{proof} The relation (i) follows from the relations 
$$ U_{0n}\approx T|_{\mathcal H_n} \buildrel d\over\prec  U_n.$$ 
Since $U_n$ are unitaries and their  spectral measures are pairwise singular,  (ii) is fulfilled. Equality (iii) is a consequence of (ii).  Also by (ii),  $$X\mathcal H_n\perp X(\vee_{k\neq n}\mathcal H_k) \text{ for all } n\geq 0.$$
Let $x\in\mathcal H$. Then $x=\mathcal Q_n x + (I_{\mathcal H}-\mathcal Q_n) x$, and
\begin{equation*}\begin{aligned}\|X|_{\mathcal H_n}\mathcal Q_n x\|^2&\leq \|X|_{\mathcal H_n}\mathcal Q_n x\|^2+
\|X|_{\vee_{k\neq n}\mathcal H_k }(I_{\mathcal H}-\mathcal Q_n) x\|^2 \\&
=\|Xx\|^2\leq\|X\|^2\|x\|^2. \end{aligned}\end{equation*}
Estimate (iv) is proved.
 \end{proof}

\begin{lemma}\label{lem2new} Suppose that $\mathcal H$ is a Hilbert space, and   
  $T\in  \mathcal L(\mathcal H)$ satisfies \eqref{main0} and \eqref{main1}. 
Let  $\{\alpha_n\}_{n=0}^\infty$ be such that $\alpha_n>0$ for all $n\geq 0$ and 
$\sum_{n=0}^\infty \alpha_n^2\|\mathcal Q_n\|^2\|X_n\|^2<\infty$. Define 
a  linear mapping  $X\colon\mathcal H\to \mathcal K_0$ by the formula $Xx_n=\alpha_n X_n x_n$, $x_n\in \mathcal H_n$, $n\geq 0$. 
Then $X$ can be  extended on $\mathcal H$ to a bounded operator such that $XT=U_0X$. 
\end{lemma}

\begin{proof}  
Let $\{x_n\}_n$ be a finite family such that $x_n\in\mathcal H_n$ for all $n\geq 0$. Then $\mathcal Q_n\sum_k x_k=x_n$ for all $n\geq 0$. 
Therefore, $\|x_n\|\leq\|\mathcal Q_n\|\|\sum_k x_k\|$. 
The boundedness of $X$ follows from the estimate
\begin{equation*}\begin{aligned} \Bigl \|X\sum_n x_n\Bigr \|^2&=\sum_n\| \alpha_n X_n x_n\|^2\leq \sum_n\alpha_n^2\| X_n\|^2\| x_n\|^2\\& \leq 
\sum_n\alpha_n^2\| X_n\|^2\|\mathcal Q_n\|^2\Bigl \|\sum_k x_k\Bigr \|^2. \end{aligned}\end{equation*}
The intertwining relation $XT=U_0X$ follows from the definition of $X$.
\end{proof}

\begin{proposition} \label{thm1} Suppose that $\mathcal H$,  $\mathcal K$ are Hilbert spaces,  $U\in\mathcal L(\mathcal K)$ is an isometry, 
  $T\in  \mathcal L(\mathcal H)$ satisfies  \eqref{sing}, \eqref{main0} and \eqref{main1}, and   
  $(X_T,U)$ is an isometric asymptote of $T$. 
Set 
\begin{equation}\label{deltan} \delta_n=\inf_{0\neq x\in\mathcal H_n}\frac{\|X_Tx\|}{\|x\|}, \ \ \ n\geq 0.\end{equation}
Then $\delta_n>0$ for all $n\geq 0$, $\mathcal K =\oplus_{n=0}^\infty X_T\mathcal H_n$, $U$ is unitary, 
and $U|_{X_T\mathcal H_n}\cong U_{0n}$ for all $n\geq 0$.
\end{proposition}

\begin{proof} Let  $\{\alpha_n\}_{n=0}^\infty$ and $X$ be from Lemma \ref{lem2new}.  
 Since $XT=U_0X$, there exists $Z\in\mathcal L(\mathcal K,\mathcal K_0)$ 
such that $X=ZX_T$ and $ZU=U_0Z$. 
In particular, 
\begin{equation} \label{yyxxttn} ZX_T x=Xx=\alpha_n X_n x \text{ for every } x\in\mathcal H_n \ \text{ and every }
n\geq 0.\end{equation} 
Let  $n\geq 0$, and let $x\in\mathcal H_n$.  
By \eqref{yyxxttn},  
$$ \|Z\|\|X_T x\|\geq\|Xx\|=\alpha_n\| X_n x\|\geq\alpha_n\frac{\|x\|}{\|X_n^{-1}\|}.$$
Thus, $\delta_n>0$. 
Therefore, $X_T\mathcal H_n$ is closed. By Lemma \ref{lem1}(ii),  $X_T\mathcal H_n$, $n\geq 0$, are pairwise orthogonal. 
By \eqref{yyxxttn}, $ZX_T\mathcal H_n=\mathcal K_{0n}$. Therefore, 
$Z|_{X_T\mathcal H_n}$ realizes the relation $U|_{X_T\mathcal H_n}\approx U_{0n}$, which implies
  $U|_{X_T\mathcal H_n}\cong U_{0n}$ by {\cite[Proposition II.3.4]{nfbk}} or {\cite[Proposition II.10.6]{conway}}. 
It follows from minimality of unitary asymptote that $\mathcal K=\oplus _{n=0}^\infty X_T\mathcal H_n$.
\end{proof}

\begin{theorem}\label{thm15} Suppose that $\mathcal H$, $\mathcal K$ are Hilbert spaces,  $U\in\mathcal L(\mathcal K)$ is unitary, 
  $T\in  \mathcal L(\mathcal H)$ satisfies  \eqref{sing},  \eqref{main0}, and \eqref{main1}, and   
 $T$ has a unitary asymptote $(X_T,U)$. 
Suppose that \begin{equation} \label{xxnxxn1}\sup_{n\geq 0}\|X_n\|\|X_n^{-1}\|<\infty. \end{equation} 

Then \begin{equation} \label{deltanqqn} \inf_{n\geq 0} \delta_n\|\mathcal Q_n\|>0, \end{equation} 
where $\delta_n$ are defined by \eqref{deltan}.  
\end{theorem}

\begin{proof} Taking into account Proposition \ref{thm1}, we can assume that $\mathcal K=\mathcal K_0$ and $U=U_0$. 
Then $X_T\mathcal H_n=\mathcal K_{0n}$.

Assume that $\inf_{n\geq 0} \delta_n\|\mathcal Q_n\|=0$. Then there exist a sequence $\{n_k\}_k$ and vectors 
$0\neq x_{n_k}\in\mathcal H_{n_k}$ such that 
\begin{equation} \label{assumcontr}\lim_k\|\mathcal Q_{n_k}\|\frac{\|X_Tx_{n_k}\|}{\|x_{n_k}\|}=0.\end{equation}
Set $\mathcal M_k=\vee_{l\geq 0}T^lx_{n_k}$. By Proposition \ref{thm1},  $X_T|_{\mathcal H_{n_k}}$ is invertible.  
Therefore,  $X_{n_k}\mathcal M_k$ and  $X_T\mathcal M_k$ are cyclic invariant subspaces for $U_{0n_k}$.  
 Both $U_{0n_k}|_{X_{n_k}\mathcal M_k}$ and $U_{0n_k}|_{X_T\mathcal M_k}$ are similar to $T|_{\mathcal M_k}$ and are unitary, 
since $U_{0n_k}$ is reductive by assumption. 
Therefore, for every $k$ there exist a finite positive Borel measure $\mu_k$ on $\mathbb T$ and   unitary transformations  
$$ W_{1k}\in\mathcal L(X_T\mathcal M_k,L^2(\mu_k)) \text{ and } W_{2k}\in\mathcal L(X_{n_k}\mathcal M_k,L^2(\mu_k))$$
such that 
$$W_{1k}U_{0n_k}|_{X_T\mathcal M_k}=U_{\mu_k}W_{1k} \text{ and } W_{2k}U_{0n_k}|_{X_{n_k}\mathcal M_k}=U_{\mu_k}W_{2k},$$
where $U_{\mu_k}$ is the operator of multiplication by the independent variable on $L^2(\mu_k)$. 
Furthermore, there exists $\varphi_k\in L^\infty(\mu_k)$ such that 
$$ \varphi_k(U_{\mu_k})=W_{1k}X_T|_{\mathcal M_k}(X_{n_k}|_{\mathcal M_k})^{-1}W_{2k}^{-1}.$$
Set $f_k=W_{2k}X_{n_k}x_{n_k}$. Then
\begin{equation} \label{phik}\frac{\|X_Tx_{n_k}\|}{\|x_{n_k}\|}\geq \frac{\|\varphi_k f_k\|}{\|X_{n_k}^{-1}\|\|f_k\|}\geq 
 \frac{\mathop{\text{\rm ess\,inf}}|\varphi_k|}{\|X_{n_k}^{-1}\|}.\end{equation} 
Take a sequence $\{\varepsilon_k\}_k$ such that $\varepsilon_k>0$ and 
$$ \varepsilon_k\|\mathcal Q_{n_k}\|\|X_{n_k}\|\to 0.$$
 It follows from \eqref{phik} that for every $k$ there exists a Borel set $\tau_k\subset\mathbb T$ such that $\mu_k(\tau_k)>0$ 
and $$|\varphi_k|\leq \|X_{n_k}^{-1}\|\frac{\|X_Tx_{n_k}\|}{\|x_{n_k}\|}+\varepsilon_k\ \ \mu_k\text{-a.e. on } \tau_k.$$
Set $\mathcal M_{1k}=X_{n_k}^{-1}W_{2k}^{-1}L^2(\tau_k,\mu_k)$. 
Since 
$$W_{1k}X_T|_{\mathcal M_{1k}}=\varphi_k(U_{\mu_k}|_{L^2(\tau_k,\mu_k)})W_{2k}X_{n_k}|_{\mathcal M_{1k}},$$
we have $$\|X_T|_{\mathcal M_{1k}}\|\leq\|\varphi_k|_{\tau_k}\|_\infty\|X_{n_k}\|\leq
\|X_{n_k}^{-1}\|\|X_{n_k}\|\frac{\|X_Tx_{n_k}\|}{\|x_{n_k}\|}+\varepsilon_k\|X_{n_k}\|.$$
Relations \eqref{xxnxxn1}, \eqref{assumcontr} and the choice of $\varepsilon_k$ imply
$$\|\mathcal Q_{n_k}\|\|X_T|_{\mathcal M_{1k}}\|\to 0.$$
There exist  sequences $\{k_j\}_j$ and $\{A_j\}_j$ such that $A_j>0$, $A_j\to\infty$, and 
\begin{equation} \label{aaj} \sum_j A_j^2 \|\mathcal Q_{n_{k_j}}\|^2\|X_T|_{\mathcal M_{1k_j}}\|^2<\infty.\end{equation}
Set $\mathcal N_j=X_{n_{k_j}}^{-1}(\mathcal K_{0n_{k_j}}\ominus X_{n_{k_j}}\mathcal M_{1k_j})$ and 
$\mathcal Q_{1j}=X_{n_{k_j}}^{-1} P_{X_{n_{k_j}}\mathcal M_{1k_j}}X_{n_{k_j}}$. Then 
 $\mathcal Q_{1j}$ is the skew  projection of $\mathcal H_{n_{k_j}}$ onto $\mathcal M_{1k_j}$ parallel to $\mathcal N_j$, 
and $\|\mathcal Q_{1j}\|\leq\|X_{n_{k_j}}^{-1}\|\|X_{n_{k_j}}\|$ for every $j$. 
 Note that $\mathcal N_j$, $\mathcal M_{1k_j}\in\operatorname{Lat}T|_{\mathcal H_{n_{k_j}}}$.

Define 
a linear mapping  $X\colon\mathcal H\to \mathcal K_0$ as follows. 
If $x\in\mathcal H_n$ with $n\neq n_{k_j}$ for all $j$, then $Xx=X_Tx$. 
If $u\in \mathcal M_{1k_j}$ for some $j$,  then $Xu= A_jX_Tu$. If $v\in\mathcal N_j$ for some $j$,   then $Xv=X_Tv$. 
Let $\{x_n\}_n$ be a finite family such that $x_n\in\mathcal H_n$ for all $n$. If $n=n_{k_j}$ for some $j$, then 
$x_n=u_j+v_j$, where $u_j\in\mathcal M_{1k_j}$ and $v_j\in\mathcal N_j$. 
Then $$(X -X_T)\Bigl(\sum_n x_n\Bigr)=\sum_j A_j X_T u_j \text{ and } u_j =
\mathcal Q_{1j}\mathcal Q_{n_{k_j}}\Bigl(\sum_n x_n\Bigr).$$
It follows from the estimate
\begin{equation*}\begin{aligned}\Bigl \|(X & -X_T)\Bigl(\sum_n x_n\Bigr)\Bigr \|^2 
 =\sum_j A_j^2\|X_T|_{\mathcal M_{1k_j}}\|^2\|u_j\|^2
\\ &
\leq \sum_j A_j^2\|X_T|_{\mathcal M_{1k_j}}\|^2\|\mathcal Q_{1j}\|^2\|\mathcal Q_{n_{k_j}}\|^2
\Bigl \|\sum_n x_n\Bigr \|^2
\\ & \leq
(\sup_{n\geq 0}\|X_n\|\|X_n^{-1}\|)^2\sum_j A_j^2\|X_T|_{\mathcal M_{1k_j}}\|^2\|\mathcal Q_{n_{k_j}}\|^2
\Bigl \|\sum_n x_n\Bigr \|^2\end{aligned}\end{equation*}
and relations \eqref{xxnxxn1} and \eqref{aaj} that $X-X_T$ can be  extended on $\mathcal H$ to a bounded operator. 
Consequently, $X$ can be  extended on $\mathcal H$ to a bounded operator. The relation $XT=U_0X$ follows from the definition of $X$. 
 Assume that $Z\in\mathcal L(\mathcal K_0)$ is such that $ZX_T=X$. Then $ZX_Tu_j=A_jX_Tu_j$ for all $u_j\in \mathcal M_{1k_j}$ and all $j$. 
Consequently, $\|Z\|\geq A_j$ for all $j$, which is impossible, because $A_j\to\infty$.

Thus, assuming that \eqref{assumcontr} is fulfilled, a contradiction is obtained. Therefore,  \eqref{deltanqqn} is proved.
\end{proof} 

\begin{theorem}\label{thm3}  Suppose that $\mathcal H$ is a Hilbert space,  
  $T\in  \mathcal L(\mathcal H)$ satisfies  \eqref{sing}, \eqref{main0} and \eqref{main1} 
 and for every (linear, bounded) transformation $X$ such that $XT=UX$ for some unitary $U$ 
\begin{equation} \label{xxnqqn} \sup_{n\geq 0}\|X|_{\mathcal H_n}\| \|\mathcal Q_n\|<\infty.\end{equation}
Suppose that there exists $X_T\in\mathcal L(\mathcal H,\mathcal K_0)$ such that $X_T T=U_0 X_T$ 
and \eqref{deltanqqn} is fulfilled (where $\delta_n$ are defined by \eqref{deltan}). 
Then 
$(X_T, U_0|_{\operatorname{clos}X_T\mathcal H})$ is a unitary asymptote of $T$.
\end{theorem}

\begin{proof}  Set $\mathcal K_{0n}'=X_T\mathcal H_n$. It follows from \eqref{deltanqqn} that $\delta_n>0$ for every $n$. Therefore, 
$\mathcal K_{0n}'$ is  closed. Furthermore, $$U_0|_{\mathcal K_{0n}'}\approx T|_{\mathcal H_n} \approx U_{0n}.$$ 
Therefore, $U_0|_{\mathcal K_{0n}'}\cong U_{0n}$ ({\cite[Proposition II.3.4]{nfbk}} or 
{\cite[Proposition II.10.6]{conway}}). Combining these relations with Lemma \ref{lem1}, 
we conclude that 
$$U_0|_{\operatorname{clos}X_T\mathcal H}=\oplus_{n=0}^\infty U_0|_{\mathcal K_{0n}'}\cong\oplus_{n=0}^\infty U_{0n}.$$
Therefore, we can assume that  $\mathcal K_{0n}'= \mathcal K_{0n}$ and 
$\operatorname{clos}X_T\mathcal H=\mathcal K_0$.

Let $\mathcal K$ be a Hilbert space, let  $U\in\mathcal L(\mathcal K)$ be unitary, and let 
$X\in\mathcal L(\mathcal H,\mathcal K)$ be  such that $XT=UX$. 

Set $Z_n=X|_{\mathcal H_n}(X_T|_{\mathcal H_n})^{-1}$. Since $\|(X_T|_{\mathcal H_n})^{-1}\|=1/\delta_n$, we have 
$$\|Z_n\|\leq\|X|_{\mathcal H_n}\|\|\mathcal Q_n\|\frac{1}{\delta_n\|\mathcal Q_n\|}.$$
By \eqref{deltanqqn} and \eqref{xxnqqn}, $\sup_n\|Z_n\|<\infty$. 

 By Lemma \ref{lem1},  the spaces $\operatorname{clos}X\mathcal H_n$, $n\geq 0$, are pairwise orthogonal.  
Set $Z=\oplus_{n=0}^\infty Z_n$. Then $Z$ is bounded, $ZU_0=UZ$, and $X=ZX_T$.  The uniqueness of $Z$ follows from the minimality of $(X_T, U_0|_{\operatorname{clos}X_T\mathcal H})$. 
\end{proof}

\begin{theorem}\label{thmmain1} Suppose that $\mathcal H$ is a Hilbert space,  and 
  $T\in  \mathcal L(\mathcal H)$ satisfies   \eqref{sing}, \eqref{main0},    \eqref{main1},  \eqref{xxnxxn1}, and  
\begin{equation} \label{qqn} \sup_{n\geq 0} \|\mathcal Q_n\|<\infty.\end{equation}
Then
$T$ has a unitary asymptote if and only if 
\begin{equation}\label{iii}\begin{aligned}&\text{there exists } c>0  \text{  such that } c^2\sum_n\| x_n\|^2\leq
\Bigl \|\sum_n x_n\Bigr \|^2 \\ &
\text{ for every  finite family } \{x_n\}_n \text{  such that } x_n\in\mathcal H_n \text{ for all  }n. \end{aligned}\end{equation} 
\end{theorem}

\begin{proof} Recall that $\delta_n$ are defined by \eqref{deltan}.  
Suppose that $X_T$ is a canonical intertwining mapping for $T$ and $U_0$ (see Proposition \ref{thm1}). 
Let $\{x_n\}_n$ be a finite family such that $x_n\in\mathcal H_n$ for all $n$. By Lemma \ref{lem1}, Theorem \ref{thm15}, and \eqref{qqn}, 
\begin{equation*}\begin{aligned} \|X_T\|^2\Bigl \|\sum_n  x_n\Bigr \|^2&\geq\Bigl \|X_T\Bigl(\sum_n  x_n\Bigr)\Bigr \|^2 = \sum_n \|X_T x_n\|^2
\\&\geq 
\sum_n\delta_n^2\|\mathcal Q_n\|^2\frac{\|x_n\|^2}{\|\mathcal Q_n\|^2}
\geq\frac{(\inf_{n\geq 0}\delta_n\|\mathcal Q_n\|)^2}{(\sup_{n\geq 0}\|\mathcal Q_n\|)^2}\sum_n\|x_n\|^2.
\end{aligned}\end{equation*} 
Conversely, if \eqref{iii} is fulfilled, 
define $X_T\in\mathcal L(\mathcal H, \mathcal K_0)$ by the formula
$$X_Tx=\frac{1}{\|X_n\|}X_n x, \ \ \ x\in \mathcal H_n, \ \ \ n\geq 0.$$
Then $\|X_T\|\leq1/c$. Furthermore, 
$$\inf_{n\geq 0}\delta_n\|\mathcal Q_n\|\geq\inf_{n\geq 0}\delta_n=\frac{1}{\sup_{n\geq 0}\|X_n\|\|X_n^{-1}\|}>0.$$
Therefore, \eqref{deltanqqn} is fulfilled. Clearly, \eqref{qqn} implies \eqref{xxnqqn}. Consequently, $X_T$ is a canonical intertwining mapping for $T$ and $U_0$ by Theorem \ref{thm3}. 
\end{proof}

\begin{lemma}\label{lempower} Suppose that $\mathcal H$ is a Hilbert space, 
  $T\in  \mathcal L(\mathcal H)$ satisfies   \eqref{sing}, \eqref{main0},  \eqref{main1},   \eqref{xxnxxn1}, and $T$ is power bounded. 
Then \eqref{qqn} is fulfilled. 
\end{lemma}

\begin{proof} By \cite{ker89}, $T$ has an isometric asymptote, and the canonical intertwining mapping $X_T$ for $T$ can 
be constructed using a Banach limit.  
In particular, for every $n\geq 0$ and every  $x\in\mathcal H$ 
\begin{equation*}\begin{aligned}  \|X_T|_{\mathcal H_n}\mathcal Q_nx\|^2&\geq\liminf_k \|(T|_{\mathcal H_n})^k\mathcal Q_nx\|^2=
\liminf_k \|X_n^{-1}U_{0n}^kX_n\mathcal Q_nx\|^2\\&\geq 
\frac{\|\mathcal Q_nx\|^2}{\|X_n^{-1}\|^2\|X_n\|^2}.\end{aligned} \end{equation*} 
By  Lemma \ref{lem1}(iv), 
 $\|X_T|_{\mathcal H_n}\mathcal Q_nx\|\leq\|X_T\|\|x\|$. Thus, 
 $$\sup_{n\geq 0} \|\mathcal Q_n\|\leq\|X_T\|\sup_{n\geq 0} \|X_n^{-1}\|\|X_n\|. \qedhere$$
\end{proof}

\begin{proposition}\label{propstar} Suppose that $\mathcal H$ is a Hilbert space,  and 
  $T\in  \mathcal L(\mathcal H)$ satisfies   \eqref{sing}, \eqref{main0},  \eqref{main1},    \eqref{xxnxxn1}, and  \eqref{qqn}. 
Furthermore, suppose that 
 \begin{equation} \label{cap} \cap_{n=0}^\infty \vee_{k=n}^\infty\mathcal H_k=\{0\}.\end{equation}
Set $\mathcal H_n'=\mathcal H\ominus  \vee_{k\neq n}\mathcal H_k$, $n\geq 0$. 
Then $T^*$ satisfies \eqref{main0},   \eqref{main1},  \eqref{xxnxxn1}, and  \eqref{qqn} with $\{\mathcal H_n'\}_{n=0}^\infty$ 
and $\{U_{0n}^{-1}\}_{n=0}^\infty$. Moreover, the following are equivalent: 
\begin{enumerate}[\upshape (i)]
\item 
there exists $c'>0$ such that $c'^2\sum_n\|x_n'\|^2\leq\|\sum_nx_n'\|^2$
for every finite family $\{x_n'\}_n$  such that $x_n'\in\mathcal H_n'$ for all $n\geq 0$; 
\item 
there exists $C>0$ such that $\|\sum_nx_n\|^2\leq C^2\sum_n\|x_n\|^2$
for every finite family $\{x_n\}_n$  such that $x_n\in\mathcal H_n$ for all $n\geq 0$. \end{enumerate}
\end{proposition}

\begin{proof}  It follows from   \eqref{main0} for  $\{\mathcal H_n\}_{n=0}^\infty$  that equality  
  \eqref{main0} for  $\{\mathcal H_n'\}_{n=0}^\infty$ is equivalent to \eqref{cap}.  
If   \eqref{main0}  and \eqref{cap} are fulfilled, then  $\mathcal Q_n^*$ is the skew projection onto $\mathcal H_n'$  parallel to $\vee_{k\neq n}\mathcal H_k'$. Thus, 
 relations \eqref{qqn} for $\{\mathcal H_n\}_{n=0}^\infty$ and $\{\mathcal H_n'\}_{n=0}^\infty$ are equivalent. 

Set $X_n'=(X_n^{-1})^*P_{\mathcal H_n}|_{\mathcal H_n'} $. Then 
$$X_n' T^*|_{\mathcal H_n'} = U_{0n}^{-1}X_n' .$$
Since $ (P_{\mathcal H_n}|_{\mathcal H_n'})^{-1}=\mathcal Q_n^*|_{\mathcal H_n}$, we have that  \eqref{main1} for $T^*$ is fulfilled, and   \eqref{xxnxxn1} for 
$\{X_n'\}_{n=0}^\infty$ follows from \eqref{xxnxxn1} for 
$\{X_n\}_{n=0}^\infty$ and  \eqref{qqn}. 

The equivalence of (i) and (ii) under \eqref{qqn} is well-known. For reader's convenience, we sketch the proof. 
Consider the mappings 
\begin{equation*}\begin{aligned}
 Y\colon\oplus_{n=0}^\infty \mathcal H_n\to\mathcal H, \ \ \ Y(\oplus_n x_n) =\sum_n x_n, \ \ \ x_n\in\mathcal H_n, \\ 
Y'\colon \mathcal H \to \oplus_{n=0}^\infty \mathcal H_n', \ \ \ Y'\sum_n x_n'=\oplus_n x_n' , \ \ \ x_n'\in\mathcal H_n', \text{ and } \\ 
 Y'_*\colon\oplus_{n=0}^\infty \mathcal H_n'\to\mathcal H, \ \ \ Y'_*(\oplus_n x_n') =\sum_n \mathcal Q_nx'_n, \ \ \ x_n'\in\mathcal H_n'. 
\end{aligned}\end{equation*}
The relations  (i) and (ii) are equivalent to the boundedness of $Y'$ and $Y$, respectively. 
If (i) is fulfilled, then  $Y'$ is bounded. Since $$((\oplus_{n=0}^\infty P_{\mathcal H_n}|_{\mathcal H_n'})Y')^* = Y,$$  
we have $Y$ is bounded. Consequently, (ii) is fulfilled. 
Conversely, suppose that (ii) is fulfilled. Then 
\begin{equation*}\begin{aligned} \|Y'_*\oplus_n x_n'\|^2 &=\|\sum_n \mathcal Q_nx'_n\|^2\leq C^2\sum_n \|\mathcal Q_nx'_n\|^2\leq  C^2 \sum_n \|\mathcal Q_n\|^2\|x'_n\|^2 \\ &
\leq C^2(\sup_{n\geq 0}\|\mathcal Q_n\|)^2 \sum_n \|x'_n\|^2.\end{aligned}\end{equation*}
Therefore, $Y'_*$ is bounded. Since $(Y'_*)^*=Y'$, we obtain that $Y'$ is bounded. Consequently, (i) is fulfilled. 
 \end{proof}

\begin{remark}\label{remapostol} If $T$ satisfies \eqref{main0}, \eqref{main1}, and \eqref{cap}, then $T\sim U_0$ by  \cite{apostol}. 
\end{remark}

\begin{corollary}\label{cor110} 
Suppose that $\mathcal H$ is a Hilbert space,  and 
  $T\in  \mathcal L(\mathcal H)$ satisfies   \eqref{sing}, \eqref{main0},  \eqref{main1},    \eqref{xxnxxn1}, \eqref{qqn}, and  \eqref{cap}.
If  $T$ and $T^*$ have unitary asymptotes, then $T\approx U_0$. 
\end{corollary}
\begin{proof} By Theorem \ref{thmmain1} applied to $T$, \eqref{iii} is fulfilled. By Proposition \ref{propstar}, $T^*$ 
satisfies the assumptions  of  Theorem \ref{thmmain1}. By Theorem \ref{thmmain1} and  Proposition \ref{propstar}, 
 relation (ii) in Proposition \ref{propstar}  is fulfilled. Define the mapping
$$X\colon\mathcal H\to\oplus_{n=0}^\infty \mathcal K_{0n}, \ \ \ X\sum_n x_n =\oplus_n\frac{X_n x_n}{\|X_n\|}, \ \ \ x_n\in\mathcal H_n.$$
By   \eqref{iii}, relation (ii) in Proposition \ref{propstar}, and \eqref{xxnxxn1}, $X$ realizes the relation  $T\approx U_0$. 
\end{proof}

\begin{corollary}\label{corpower} 
Suppose that $\mathcal H$ is a Hilbert space,  and 
  $T\in  \mathcal L(\mathcal H)$ is a power bounded operator which  satisfies   \eqref{sing}, \eqref{main0},  \eqref{main1},    \eqref{xxnxxn1},  and  \eqref{cap}.
Then $T\approx U_0$. 
\end{corollary}
\begin{proof} The conclusion of the
 corollary follows from Lemma \ref{lempower}, Corollary \ref{cor110} and \cite{ker89}.  
\end{proof}

\begin{remark} Let $\{U_{0n}\}_{n=0}^\infty$ be a  family of unitary operators  satisfying  \eqref{sing}. 
Denote the scalar spectral measures of $U_{0n}$ by $\nu_n$, $n\geq 0$.  
 If there exists $n\geq 0$ such that  $U_{0n}$ is not reductive, then $m$ is absolutely continuous with respect to $\nu_n$ 
 (where $m$ is the Lebesgue measure on $\mathbb T$). Therefore,  $\nu_k$ is singular with respect to $m$ for every $k\geq 0$,  $k\neq n$.   Consequently, $U_{0k}$ are reductive for all $k\geq 0$,  $k\neq n$. By \eqref{main0} and \eqref{main1}, 
$T\approx U_{0n}\oplus T|_{\vee_{k\neq n}\mathcal H_k}$. By  {\cite[Proposition 13]{kerntuple}}, the existence of unitary asymptotes of $T$ and $T^*$ is equivalent to  the existence of unitary asymptotes of 
$T|_{\vee_{k\neq n}\mathcal H_k}$ and $(T|_{\vee_{k\neq n}\mathcal H_k})^*$, respectively.
\end{remark}

\section{Operators quasisimilar to cyclic unitaries with pure atomic spectrum} 

In this section, a simplest particular case of operators from the previous section is considered.  Namely, it is assumed that   $\dim\mathcal H_n=1$ for all $n\geq 0$, where the  spaces  $\mathcal H_n$ are from \eqref{main0}. 
Examples of operators which are  quasisimilar to  cyclic unitaries with pure atomic spectrum and does not have a unitary asymptote are given. 

Let $\mathcal H$ be a Hilbert space, and let 
$\{x_n\}_{n=0}^\infty\subset\mathcal H$ be such that $\vee_{n=0}^\infty x_n=\mathcal H$ and
 $x_n\neq 0$ for all $n\geq 0$. Set $\mathcal H_n=\mathbb C x_n$. 
 Relations  \eqref{sing} and \eqref{main1} mean that there exist
$\{\lambda_n\}_{n=0}^\infty\subset\mathbb T$ and  an orthonormal basis $\{e_n\}_{n=0}^\infty$ of $ \mathcal K_0$ 
such that $\lambda_n\neq\lambda_k$ for $n\neq k$, $n,k\geq 0$, 
$U_0\in\mathcal L(\mathcal K_0)$ acts by the formula $U_0e_n=\lambda_ne_n$, $n\geq 0$,  and 
$Tx_n=\lambda_nx_n$ for all $n\geq 0 $. Relation \eqref{xxnxxn1} is fulfilled automatically, 
and estimate \eqref{xxnqqn} is fulfilled by Lemma \ref{lem1}(iv). 

Relation \eqref{main0} is equivalent to the existence of a family $\{x'_n\}_{n=0}^\infty\subset \mathcal H$ such that 
\begin{equation}\label{main01} (x'_n,x_n)=1 \ \ \text{ and } 
  (x'_n,x_k)=0, \ \ \text{ if }  k\neq n, \ n,k\geq 0,\end{equation} 
which is equivalent that $\mathcal Q_n$ acts by the formula 
 \begin{equation}\label{qqndef} \mathcal Q_n=x_n\otimes x'_n \text{ for every }n\geq 0.\end{equation} 
Clearly, $\|\mathcal Q_n\|=\|x_n\|\|x'_n\|$. 
It is easy to see that $T^*x'_n=\overline\lambda_nx'_n$, $n\geq 0$. 
Under previous assumptions,  \eqref{cap} is equivalent to the relation
\begin{equation}\label{cap1}\mathcal H=\vee_{n\geq 0}x'_n. \end{equation} 
 
For $n\geq 0$, set
\begin{equation}\label{ppndef} \mathcal P_n=\sum_{k=0}^n \mathcal Q_k. \end{equation}  
Clearly, the relation  \begin{equation}\label{ppnest}\sup_{n\geq 0}\|\mathcal P_n\|<\infty \end{equation}
implies \eqref{qqn}. Furthermore, \eqref{ppnest} implies that 
 $$x=\lim_n\mathcal P_nx=\sum_{k\geq 0}(x,x'_k)x_k \ \ \text{ for every } x\in\mathcal H.$$
In particular, \eqref{cap1} follows from  \eqref{ppnest}. 
For references, see 
{\cite[Ch. VI.1,VI.3]{nik86}}, {\cite[Ch. I.A.5.1,  II.C.3.1]{nik02}}, 
or {\cite[Theorems I.6.1, I.7.1, p. 53, 58, and Proposition II.II.14.1, p. 427]{singer}}
 
The proof of the following known lemma is given for reader's convenience. 

\begin{lemma} \label{lemttbounded} Suppose that $\mathcal H$ is a Hilbert space,  
$\{x_n\}_{n=0}^\infty\subset\mathcal H$ is such that $\vee_{n=0}^\infty x_n=\mathcal H$ and
 $x_n\neq 0$ for all $n\geq 0$, \eqref{main01} and   \eqref{ppnest} are fulfilled.  
Let $0<\ldots <t_{n+1}<t_n<\ldots<t_0\leq 2\pi$. Set $\lambda_n=e^{it_n}$ and $Tx_n=\lambda_n x_n$, $n\geq 0$. 
Then the linear mapping $T$ can be boundedly extended on $\mathcal H$. 
\end{lemma} 
\begin{proof}  Let $N<\infty$, and let $x\in\vee_{n=0}^N x_n$. 
We have  
$$Tx=\lambda_0\mathcal P_0 x + \sum_{k=1}^N \lambda_k(\mathcal P_k-\mathcal P_{k-1})x=
 \sum_{k=0}^{N-1}( \lambda_k-\lambda_{k+1})\mathcal P_k x +\lambda_N\mathcal P_N x.$$
Therefore, \begin{align*}\|Tx\|&\leq \sum_{k=0}^{N-1}|\lambda_k-\lambda_{k+1}|\|\mathcal P_k\|\| x \|+|\lambda_N|\|\mathcal P_N\|\| x\|\\&
\leq\sup_{n\geq 0}\|\mathcal P_n\|\Bigl(\sum_{k=0}^{N-1}|\lambda_k-\lambda_{k+1}|+|\lambda_N|\Bigr)\| x\|\\&\leq
(2\pi+1)\sup_{n\geq 0}\|\mathcal P_n\|\| x\|. \qedhere\end{align*}
\end{proof}

The family  $\{x_n\}_{n=0}^\infty$ is called a \emph{Riesz basis} of $\mathcal H$, if a linear mapping $W$ acting by the formula $We_n=x_n$ for an orthonormal basis $\{e_n\}_{n=0}^\infty$ is  bounded and has bounded inverse. 

Before the formulation of Lemma \ref{lempsi} we recall some definitions and statements. 

Let $v\in L^2(\mathbb T, m)$ be a real-valued function. Then there exists a real-valued function $\widetilde v\in L^2(\mathbb T, m)$ 
such that $\int_{\mathbb T}\widetilde v\mathrm{d}m=0$ and $v+i\widetilde v\in H^2$. The function $\widetilde v$ is called 
the (\emph{harmonic}) \emph{conjugate} of $v$ (see, for example, {\cite[Definition I.A.5.3.2]{nik02}} or {\cite[Sec. 2.1]{borsp}}). 

Let $w\in L^1(\mathbb T, m)$, and let $w\geq 0$. The function $w$ is called a \emph{Helson--Szeg\H o weight function}, if $w=e^{u+\widetilde v}$, where
 $u$, $v\in L^\infty(\mathbb T, m)$ are real-valued functions, and $\|v\|_\infty<\pi/2$. If $\psi\in H^2$, $\psi$ is outer, and  
$\psi=|\psi|e^{i\phi}$ $m$-a.e. on $\mathbb T$, where $\phi\in L^\infty(\mathbb T, m)$ is a real-valued function  and $\|\phi\|_\infty<\pi/4$, 
then $|\psi|^2$ is a Helson--Szeg\H o weight function. For a proof, see {\cite[Theorem 2.3(i)]{borsp}}, 
or {\cite[Theorems I.A.5.4.1 and I.B.4.3.1, Lemma I.B.4.3.3, Corollary I.B.4.2.6(2)]{nik02}}, 
or {\cite[Ch. VIII.6 and Appendix 4, Theorem 36, Corollaries 31 and 38]{nik86}}. 

The following theorem is a part of the Helson--Szeg\H o theorem, which will be used in the proof of Lemma \ref{lempsi}.

\begin{theoremcite}[Helson--Szeg\H o theorem
 {\cite[Theorem 2.1]{borsp}}, {\cite[Ch. VIII.6]{nik86}}, {\cite[Lemma I.A.5.2.5, Theorem I.A.5.4.1]{nik02}}]\label{thmhhss}
Let $w\in L^1(\mathbb T, m)$, and let $w\geq 0$. Set $\chi(\zeta)=\zeta$ $(\zeta\in\mathbb T)$. For $n\geq 0$, 
define the mapping $\mathcal P_{w,n}$ on the linear set 
$$\Bigl\{\sum_{k=-N}^N a_k \chi^k \ :\  \{a_k\}_{k=-N}^N\subset \mathbb C, N=0,1,\ldots\Bigr\}\subset 
L^2(\mathbb T, w\mathrm{d}m)$$ by the formula
 \begin{equation*}\mathcal P_{w,n}\sum_k a_k \chi^k = \sum_{k=-n}^n a_k \chi^k.\end{equation*}
Then the following are equivalent:
\begin{enumerate}[\upshape (i)]
\item $w$ is a Helson--Szeg\H o weight function;

\item $\mathcal P_{w,n}$ can be extended on $L^2(\mathbb T, w\mathrm{d}m)$ to bounded operators for all $n\geq 0$, and 
$$\sup_{n\geq 0}\|\mathcal P_{w,n}\|<\infty.$$
\end{enumerate}
\end{theoremcite}

\begin{lemma} 
\label{lempsi}
Suppose that $\psi\in H^2$ is an outer function such that $|\psi|^2$ is a Helson--Szeg\H o weight function. 
Consider $H^2$ and $x_n=\chi^n\psi$, $n\geq 0$.  Then there exists a family  $\{x'_n\}_{n=0}^\infty\subset H^2$ such that  \eqref{main01} and \eqref{cap1} are fulfilled,  and
$\{\mathcal P_n\}_{n=0}^\infty$ defined by \eqref{ppndef} satisfies \eqref{ppnest}. Moreover,  $\{x_n\}_{n=0}^\infty$ is a Riesz basis of  $H^2$ if and only if  $\psi$, $1/\psi\in H^\infty$.
\end{lemma}
\begin{proof}
Set $$ P^2(w\mathrm{d}m)=\operatorname{clos}_{L^2(\mathbb T, w\mathrm{d}m)}
\Bigl\{\sum_{k=0}^N a_k \chi^k \ :\  \{a_k\}_{k=0}^N\subset \mathbb C, N=0,1,\ldots\Bigr\}$$ 
and define the mapping $J_\psi\colon P^2(w\mathrm{d}m)\to H^2$ by the formula
$$ J_\psi\sum_{k=0}^N a_k \chi^k=\sum_{k=0}^N a_k x_k.$$ 
Then $J_\psi$ can be extended to a unitary transformation. Set $$\mathcal P_n=J_\psi\mathcal P_{w,n}|_{P^2(w\mathrm{d}m)}J_\psi^{-1}.$$ 
Estimate \eqref{ppnest} follows from the latter equality and the Helson--Szeg\H o theorem, see Theorem \ref{thmhhss}. 
Note that 
$$\mathcal P_n\sum_{k=0}^N a_k x_k=\sum_{k=0}^n a_k x_k,\  \  \{a_k\}_{k=0}^N\subset \mathbb C,
 N=0,1,\ldots .$$ 
Set $\mathcal Q_0=\mathcal P_0$ and $\mathcal Q_n=\mathcal P_n-\mathcal P_{n-1}$ for $n\geq 1$. Clearly, 
$\mathcal Q_n$ are one-dimensional operators, therefore, there exists a family  $\{x'_n\}_{n=0}^\infty\subset H^2$ such that 
\eqref{qqndef} is fulfilled. 
Clearly, $\{\mathcal P_n\}_{n=0}^\infty$  satisfies \eqref{ppndef}.  Therefore, \eqref{main01} and \eqref{cap1}
follow from   \eqref{ppnest}, see the references after  \eqref{ppnest}.

Clearly, $\{\chi^n\}_{n=0}^\infty$  is an orthonormal basis of $H^2$. 
Define $M_\psi\colon H^2\to H^2$ by the formula $M_\psi h=\psi h$ ($h\in H^2$). Clearly,  $M_\psi\chi^n= x_n$, $n\geq 0$.  Therefore, $\{x_n\}_{n=0}^\infty$ is a Riesz basis of $H^2$ if and only if  $M_\psi$ is bounded and has bounded inverse on $H^2$, which is equivalent to the inclusions $\psi$, $1/\psi\in H^\infty$.  
 \end{proof}

In Example \ref{exapsi} below, Lemma \ref{lempsi} and examples from {\cite[Ch. I.A.5]{nik02}}, 
 {\cite[Sec. 2.1, Example 3.3.2]{borsp}}, {\cite[Example II.I.11.2, p. 351]{singer}}, {\cite[Ch. VIII.6]{nik86}} are used. 

\begin{example}\label{exapsi}
For $\alpha\in(-1/2,0)\cup(0,1/2)$ set $\psi_\alpha(z)=(1-z)^\alpha$, $z\in\mathbb D$.
Then $\psi_\alpha$ is a Helson--Szeg\H o weight function. Therefore, for $\mathcal H=H^2$ and $x_n=\chi^n\psi_\alpha$,  $n\geq 0$, 
there exists $x'_n$ such that \eqref{main01} and \eqref{ppnest} (and, consequently, \eqref{qqn} and \eqref{cap1}) are fulfilled. 
But $\{x_n\}_{n=0}^\infty$ is not a Riesz basis.
For $\{\lambda_n\}_{n=0}^\infty$ from Lemma \ref{lemttbounded} 
define $U_0$, $T\in\mathcal L(H^2)$ by the formulas $U_0 \chi^n=\lambda_n\chi^n$, $Tx_n=\lambda_nx_n$, $n\geq 0$.  
Consider the linear mappings $X_\alpha$ and $X_{*\alpha}$ acting by the formulas 
$X _\alpha x_n=\chi^n$ and $X_{*\alpha}\chi^n=x_n$, $n\geq 0$. Applying Theorems \ref{thm3}, \ref{thmmain1}, 
Proposition \ref{propstar} and Corollary \ref{cor110},  we obtain the following. 
If $\alpha\in(-1/2,0)$, then 
$X_\alpha$ is bounded and is the canonical intertwining mapping  for $T$ and $U_0$, and $T^*$ does not have a unitary asymptote. 
If $\alpha\in(0,1/2)$, then 
$X_{*\alpha}$ is bounded,  $X_{*\alpha}^*$  is the canonical intertwining mapping  for $T^*$ and $U_0^*$, and $T$ 
does not have a unitary asymptote. By Remark \ref{remapostol}, $T\sim U_0$. 
\end{example}

\begin{example}\label{exanet} Let $T_1$ and $T_2$ be constructed as in Example \ref{exapsi} and have the following properties: 
$T_1$ and $T_2^*$ have unitary asymptotes,  $T_1^*$ and $T_2$ do not have a unitary asymptote. Set $T=T_1\oplus T_2$. By  {\cite[Proposition 13]{kerntuple}}, 
$T$ and $T^*$ do not have a unitary asymptote. 
\end{example}

In Examples \ref{exapsi} and \ref{exanet} the families $\{x_n\}_{n=0}^\infty$, $\{x'_n\}_{n=0}^\infty$ satisfying  \eqref{qqn} 
(where $\mathcal Q_n$, $n\geq 0$,  are defined by \eqref{qqndef}) 
are considered. In Example \ref{exanoest} the operator $T$ is constructed such that $T$ and $T^*$ have unitary asymptotes 
and \eqref{qqn} is not fulfilled. (Since \eqref{qqn} is not fulfilled, $T\not\approx U_0$.)

\begin{example}\label{exanoest}
 Let $\{e_n\}_{n=0}^\infty$ be an orthonormal basis of a Hilbert space $\mathcal H$, and let $\{c_n\}_{n=0}^\infty$ be such that 
 $c_n>0$ and $c_n\to\infty$. Set $x_{2n}=e_{2n}$, $x_{2n+1}=e_{2n+1}+c_ne_{2n}$, then $x'_{2n}=e_{2n}-c_n e_{2n+1}$,  
$x'_{2n+1}=e_{2n+1}$, $n\geq 0$. We have $\|x_n\|\|x'_n\|=(1+c_n^2)^{1/2}\to\infty$. Thus, \eqref{qqn} is not fulfilled (where $\mathcal Q_n$, $n\geq 0$,  are defined by \eqref{qqndef}). 

Let $\{\lambda_n\}_{n=0}^\infty\subset \mathbb T$ be such that 
$\sup_{n\geq 0}c_n|\lambda_{2n}-\lambda_{2n+1}|<\infty$. Define a linear mapping $T$ by the formula 
 $Tx_n=\lambda_nx_n$, $n\geq 0$. We will to prove that $T$ is bounded, and $T$ and $T^*$ have unitary asymptotes. 

Set $\mathcal H_0=\vee_{n\geq 0} e_{2n}$ and   $\mathcal H_1=\vee_{n\geq 0} e_{2n+1}$. It follows from the definition of $T$ 
 that to prove the boundedness of $T$ it is  sufficient to prove the boundedness of $P_{\mathcal H_0}T|_{\mathcal H_1}$. 

Let $\{a_n\}_{n=0}^\infty\subset \mathbb C$ be such that $\sum_{n\geq 0}|a_n|^2<\infty$. 
Since $e_{2n+1}=x_{2n+1}-c_ne_{2n}$, we have
$$ P_{\mathcal H_0}T\sum_{n\geq 0}a_ne_{2n+1}=\sum_{n\geq 0}a_nc_n(\lambda_{2n+1}-\lambda_{2n})e_{2n}. $$
Therefore,
 \begin{align*}\Bigl\|P_{\mathcal H_0}T\sum_{n\geq 0}a_ne_{2n+1}\Bigr\|^2&=\sum_{n\geq 0}|a_nc_n(\lambda_{2n+1}-\lambda_{2n})|^2\\&\leq
(\sup_{n\geq 0}c_n|\lambda_{2n}-\lambda_{2n+1}|)^2\sum_{n\geq 0}|a_n|^2.\end{align*} 
The boundedness of $T$ is proved. 

Take $\{\alpha_n\}_{n=0}^\infty$, $\{\alpha'_n\}_{n=0}^\infty$ such that $\alpha_n>0$, $\alpha'_n>0$ for all $n\geq 0$, 
\begin{equation}\label{sup} \sup_{n\geq 0}\alpha_n<\infty, \ \ \sup_{n\geq 0}\alpha'_n<\infty, \ \ \sup_{n\geq 0}c_n\alpha_{2n}<\infty,  
\ \ \sup_{n\geq 0}c_n\alpha'_{2n+1}<\infty,\end{equation}
\begin{equation}\label{inf}\begin{gathered}\inf_{n\geq 0}\alpha_{2n+1}>0,  \ \ \inf_{n\geq 0}\alpha'_{2n}>0,  \\
 \inf_{n\geq 0}(1+c_n^2)^{1/2}\alpha_{2n}>0, 
\ \ \inf_{n\geq 0}(1+c_n^2)^{1/2}\alpha'_{2n+1}>0.\end{gathered}\end{equation}
Define $X$ and $X_*$ by the formulas $Xx_n=\alpha_ne_n$, $X_*x'_n=\alpha'_ne_n$, $n\geq 0$. Then $X$ and $X_*$ are canonical intertwining mappings  
for $T$ and $T^*$, respectively. Indeed,  the boundedness of $X$ and $X_*$ is proved using \eqref{sup} 
similarly to the proving of the boundedness of $T$. When boundedness was proved,  Lemma \ref{lem1}(iv), Theorem \ref{thm3} and \eqref{inf} are applied.   

By Remark \ref{remapostol}, $T\sim U_0$, where $U_0e_n=\lambda_n e_n$, $n\geq 0$. 
\end{example}

\section{Operators quasisimilar to $S$}

Recall that $H^2$ is the Hardy space on the unit disc $\mathbb D$ and 
$m$ is the normalized Lebesgue (arc length) measure  on the unit circle $\mathbb T$. 
Set $\chi(z)=z$, $z\in\overline{\mathbb D}$. Then $S\in \mathcal L(H^2)$ is the operator of multiplication by $\chi$. 
Functions $h$ from $H^2$ have nontangential boundary values $h(\zeta)$ for  $m$-a.e. $\zeta \in\mathbb T$. 
Let ${\mathbf u}$ be an inner function, that is, $|{\mathbf u}(\zeta)|=1$ for  $m$-a.e. $\zeta \in\mathbb T$. 
Put $\mathcal K_{\mathbf u}=H^2\ominus{\mathbf u} H^2$. The space $\mathcal K_{\mathbf u}$ is called the 
\emph{model space} 
corresponding to ${\mathbf u}$. It is well known and easy to see that $f\in\mathcal K_{\mathbf u}$ if and only if 
${\mathbf u}\overline\chi\overline f\in\mathcal K_{\mathbf u}$. Let  $\zeta\in\mathbb T$ be such that 
 the nontangential boundary value ${\mathbf u}(\zeta)$ at  $\zeta$ exists and  $|{\mathbf u}(\zeta)|=1$. Define a function  
\begin{equation}
\label{y1}
k_{{\mathbf u},\zeta}(z)=\frac{1-\overline{{\mathbf u}(\zeta)}{\mathbf u}(z)}{1-\overline\zeta z},\ \ \ z\in\mathbb D,  
\end{equation}
which is  analytic in $\mathbb D$. If $k_{{\mathbf u},\zeta}\in H^2$, then 
$k_{{\mathbf u},\zeta}\in\mathcal K_{\mathbf u}$, every $f\in\mathcal K_{\mathbf u}$ has nontangential boundary value $f(\zeta)$ at $\zeta$, and $f(\zeta)=(f,k_{{\mathbf u},\zeta})$. In particular, if $\mathbf u$ has analytic continuation 
in a neighborhood of $\zeta\in\mathbb T$, then $k_{{\mathbf u},\zeta}\in H^2$.  For references see, for example, 
{\cite[Theorem 8.6.1]{cauchy}} or {\cite[Theorem 6.11]{model}}.  

The space $\mathcal K_{\mathbf u}$ is a coinvariant subspace of $S\in\mathcal L(H^2)$. 
 It is easy to see that  
\begin{equation}
\label{y2}
P_{\mathcal K_{\mathbf u}}S f =\chi f-(f,P_{H^2}({\mathbf u}\overline\chi)){\mathbf u}, \ \ \ f\in\mathcal K_{\mathbf u}. 
\end{equation}

Let $\mathbf u$, $\mathbf v$ be two inner function.  An analytic  function $\varphi\colon\mathbb D\to\mathbb C$ is called a   multiplier between $\mathcal K_{\mathbf u}$ and $\mathcal K_{\mathbf v}$, if 
$\varphi\mathcal K_{\mathbf u}\subset\mathcal K_{\mathbf v}$. Multipliers between model spaces are studied in  \cite{c} and  \cite{fhr}. 

\begin{proposition}\label{propdeftt} Let $\mathbf u$, $\mathbf v$ be two inner function, and let $\varphi_0\in H^\infty$ 
be a multiplier between $\mathcal K_{\mathbf u}$ and $\mathcal K_{\mathbf v}$ such that  
 \begin{equation}\label{closvarphi0}\operatorname{clos}\varphi_0\mathcal K_{\mathbf u}=\mathcal K_{\mathbf v}.
\end{equation} Then $\varphi_0$ is outer, and 
 $\varphi_0=c{\mathbf v}\overline{\mathbf u}\overline \varphi_0$ $m$-a.e. on $\mathbb T$ for some $c\in\mathbb T$. 
Define $T$, $X$, $Y\in\mathcal L(H^2)$ by the formulas
$$ Y({\mathbf u}h+f)={\mathbf v}h+\varphi_0 f, \ \ \ X({\mathbf v}h+g)={\mathbf u}\varphi_0 h+g, \ \ \ \ h\in H^2, \ f\in \mathcal K_{\mathbf u}, \ 
g\in\mathcal K_{\mathbf v}, $$
$$T=S+({\mathbf v}-\varphi_0{\mathbf u})\otimes\frac{1}{\overline c\varphi_0(0)}P_{H^2}({\mathbf v}\overline\chi).$$
Then $X$ and $Y$ are quasiaffinities such that $YS=TY$, $XT=SX$, and $XY=\varphi_0(S)$. 

Moreover, if $Z\in\mathcal L(H^2)$ is such that $ZT=SZ$, then there exists $\psi\in H^\infty$ such that
\begin{equation}\label{zzyy} ZY=\psi(S),\end{equation} 
\begin{equation}\label{psiv} \frac{\psi}{\varphi_0}\mathcal K_{\mathbf v}\subset H^2,
\end{equation}
and 
\begin{equation}\label{zzdef} Z({\mathbf v}h+g)={\mathbf u}\psi h+\frac{\psi}{\varphi_0}g, 
\ \ \ h\in H^2, \ g\in\mathcal K_{\mathbf v}. 
\end{equation}
Conversely, let $\psi\in H^\infty$  satisfy \eqref{psiv}. Define  $Z$  by \eqref{zzdef}. Then 
 $Z\in\mathcal L(H^2)$, \eqref{zzyy} is fulfilled,  and $ZT=SZ$. 
\end{proposition}

\begin{proof} Many statements of the proposition can be easily checked directly or using  \cite{c} and  \cite{fhr}. 
The equality $\ker Y=\{0\}$ is evident from the definition of $Y$. The equality  
\begin{equation}\label{closyy}\operatorname{clos}YH^2=H^2\end{equation} follows from 
\eqref{closvarphi0}.  
We will to  prove that $YS=TY$. Note that the latest equality is equivalent to the equality
$$ YS-SY=({\mathbf v}-\varphi_0{\mathbf u})\otimes\frac{1}{\overline c\varphi_0(0)}Y^*P_{H^2}({\mathbf v}\overline\chi).$$
Let $h\in H^2$, and let $f\in \mathcal K_{\mathbf u}$. Then, by \eqref{y2},  
\begin{equation*}\begin{aligned}
(YS-SY)({\mathbf u}h+f)& =Y\bigl(\chi{\mathbf u}h+(f,P_{H^2}({\mathbf u}\overline\chi)){\mathbf u}+
P_{\mathcal K_{\mathbf u}}(\chi f)\bigr) - \chi{\mathbf v}h-\chi\varphi_0 f \\&
= (f,P_{H^2}({\mathbf u}\overline\chi)){\mathbf v} +\bigl(\chi f-(f,P_{H^2}({\mathbf u}\overline\chi)){\mathbf u}\bigr)\varphi_0-\chi\varphi_0 f\\&
=(f,P_{H^2}({\mathbf u}\overline\chi))({\mathbf v}-\varphi_0 {\mathbf u}).\end{aligned}\end{equation*}
If $h\in H^2$ and $ g\in\mathcal K_{\mathbf v}$, then $Y^*({\mathbf v}h+g)={\mathbf u}h+P_{\mathcal K_{\mathbf u}}\overline\varphi_0 g$. 
In particular, $Y^*P_{H^2}({\mathbf v}\overline\chi)=\overline c\varphi_0(0)P_{H^2}({\mathbf u}\overline\chi)$. Thus, the equality $YS=TY$ is proved. 

Now we  will to prove that $\ker X=\{0\}$. 
Assume that $h\in H^2$ and $g\in\mathcal K_{\mathbf v}$ are such that $X({\mathbf v}h-{\mathbf v}\overline\chi\overline g)=0$. 
This means that 
${\mathbf u}\varphi_0 h={\mathbf v}\overline\chi\overline g$. 
Since $\varphi_0=c{\mathbf v}\overline{\mathbf u}\overline \varphi_0$  $m$-a.e. on $\mathbb T$ , we have 
$$\frac{\overline\chi\overline g}{\overline\varphi_0}= c h\ \  \  m\text{-a.e. on  }\mathbb T.$$
Since $\varphi_0$ is outer, the latest equality implies that 
$$\frac{\overline\chi\overline g}{\overline\varphi_0}\in \overline\chi\overline{H^2}\cap H^2=\{0\}.$$
Thus, $g=0$ and $h=0$.   

Note that $XTY=XYS=\varphi_0(S)S= S\varphi_0(S)=SXY$. It follows from the latest equality and \eqref{closyy} that $XT=SX$.

Let $Z\in\mathcal L(H^2)$ be such that $ZT=SZ$. Then $ZYS=ZTY=SZY$. Therefore, there exists $\psi\in H^\infty$ such that \eqref{zzyy} is fulfilled. It follows from \eqref{zzyy}  and the definition of $Y$ that $Z\varphi_0f=\psi f$ for every $f\in \mathcal K_{\mathbf u}$. The latest relation, the boundedness of $Z$ and  \eqref{closvarphi0} imply  \eqref{psiv}. 
Equality \eqref{zzdef} follows from the definition of $Y$ and \eqref{zzyy}. 

Conversely, if $\psi\in H^\infty$  satisfies \eqref{psiv} and  $Z$ is defined by \eqref{zzdef}, 
then $Z$ is bounded by the Closed Graph Theorem.   Equality \eqref{zzyy} follows  from \eqref{zzdef}  and implies  $ZTY=ZYS=SZY$. 
 It follows from the latest equality and \eqref{closyy} that $ZT=SZ$.
\end{proof}

\begin{lemma}\label{lemssmain} Suppose that $\mathbf v$ is an inner function. Set  
\begin{equation}\label{eev} E_{\mathbf v}=\{\zeta\in\mathbb T\ : \ k_{\mathbf v,\zeta}\in H^2\},
\end{equation}
where $k_{\mathbf v,\zeta}$ are defined by \eqref{y1} for $\mathbf v$. 
Let $m(E_{\mathbf v})=1$. Suppose that $\varphi_0\in H^\infty$ is outer, 
 $1/\varphi_0\not\in H^2$,
 and $\varphi_1\in H^2$ is such that $\varphi_0\varphi_1\in H^\infty$ and 
$\varphi_1\mathcal K_{\mathbf v}\subset H^2$.
Then there exists $\varphi\in H^2$ such that 
$$\varphi\varphi_0\in H^\infty, \ \ \ \varphi\mathcal K_{\mathbf v}\subset H^2, \ \ \ 
 \ \text{ and } \ \ \frac{\varphi}{\varphi_1}\not\in H^\infty.$$
\end{lemma}

\begin{proof} Note that $\mathop{\text{\rm ess\,inf}}_{\mathbb T}|\varphi_0||\varphi_1|=0$. 
Indeed, if there exists $\delta>0$ such that $|\varphi_0||\varphi_1|\geq \delta$  $m$-a.e. on $\mathbb T$, 
then $|\varphi_1|\geq \delta/|\varphi_0|$  $m$-a.e. on $\mathbb T$. Since $\varphi_0$ is outer, 
the latest estimate implies that $1/\varphi_0\in H^2$, which contradicts to the assumption on $\varphi_0$. 

Take $\{C_n\}_{n=1}^\infty$ such that $0<C_n<C_{n+1}$ for all $n\geq 1$ and $C_n\to\infty$. 
Set $$E_n=\{\zeta\in E_{\mathbf v}\ :\ C_n<\|k_{\mathbf v,\zeta}\|\leq C_{n+1}\}, \ \ n \geq 1,$$
and $$\delta_n =\mathop{\text{\rm ess\,inf}}_{E_n}|\varphi_0||\varphi_1|,\ \ n \geq 1.$$
(If $m(E_n)=0$, then $\delta_n=+\infty$.)

Consider two cases. \emph{First case}: there exists $n\geq 1$ such that $\delta_n=0$. Then 
there exist sequences $\{\delta_{1k}\}_{k=1}^\infty$ and $\{\tau_{1k}\}_{k=1}^\infty$ such that
 $0<\delta_{1k+1}<\delta_{1k}$, $\delta_{1k}\to 0$, $\tau_{1k}\subset E_n$, $m(\tau_{1k})>0$, 
and $$\delta_{1k+1}<|\varphi_0||\varphi_1|\leq\delta_{1k} \  m\text{-a.e. on }\tau_{1k} \text{ for all }k\geq 1. $$
Set $A_k=1/\delta_{1k}$, $k\geq 1$. Take a sequence $\{\varepsilon_k\}_{k=1}^\infty$ such that 
$\varepsilon_k>0$  for all $k\geq 1$, and 
\begin{equation}\label{aaepsilon} \sum_{k=1}^\infty\varepsilon_kA_k^2<\infty.\end{equation}
There  exists a sequence $\{\tau_k\}_{k=1}^\infty$ such that $\tau_k\subset\tau_{1k}$, $m(\tau_k)>0$, 
and $$\int_{\tau_k}|\varphi_1|^2{\mathrm d}m\leq\varepsilon_k \ \text{ for all }k\geq 1.$$
Then \begin{equation}\label{aavarphi1}  \sum_{k=1}^\infty A_k^2\int_{\tau_k}|\varphi_1|^2{\mathrm d}m<\infty\end{equation}
and for every $g\in\mathcal K_{\mathbf v}$ 
\begin{equation*}\begin{aligned} \sum_{k=1}^\infty   A_k^2\int_{\tau_k}|\varphi_1(\zeta)|^2|g(\zeta)|^2{\mathrm d}m(\zeta)  &
\leq \|g\|^2\sum_{k=1}^\infty A_k^2\int_{\tau_k}|\varphi_1(\zeta)|^2\|k_{\mathbf v,\zeta}\|^2{\mathrm d}m(\zeta) \\& 
\leq C_{n+1}^2\|g\|^2\sum_{k=1}^\infty A_k^2\int_{\tau_k}|\varphi_1(\zeta)|^2{\mathrm d}m(\zeta),\end{aligned}\end{equation*}
since $\tau_k\subset E_n$ for all $k\geq 1$. 

\emph{Second case}: $\delta_n>0$ for every $n\geq 1$.   Then there exists a subsequence $\{\delta_{n_k}\}_{k=1}^\infty$ such that 
 $0<\delta_{n_{k+1}}<\delta_{n_k}$ for all $k\geq 1$ and $\delta_{n_k}\to 0$. 
Set $$\tau_{1k}=\{\zeta\in E_{n_k}\ : \ |\varphi_0(\zeta)||\varphi_1(\zeta)|\leq\delta_{n_{k-1}}\}, \ \ \ k\geq 1.$$
Then $m(\tau_{1k})>0$ for every $k\geq 1$. Set $A_k=1/\delta_{n_{k-1}}$, $k\geq 1$. 
 Take a sequence $\{\varepsilon_k\}_{k=1}^\infty$ such that 
$\varepsilon_k>0$ for all $k\geq 1$, and \eqref{aaepsilon} is fulfilled. 
There  exists a sequence $\{\tau_k\}_{k=1}^\infty$ such that $\tau_k\subset\tau_{1k}$, $m(\tau_k)>0$, 
and $$C_{n_k+1}^2\int_{\tau_k}|\varphi_1|^2{\mathrm d}m\leq\varepsilon_k \ \text{ for all }k\geq 1.$$
Then \eqref{aavarphi1}  is fulfilled. 
Let $g\in\mathcal K_{\mathbf v}$. Since $\tau_k\subset E_{n_k}$, we have
$$ |g(\zeta)|=|(g,k_{\mathbf v,\zeta})|\leq\|g\|\|k_{\mathbf v,\zeta}\|\leq\|g\|C_{n_k+1}
\ \ \text{ for all }\zeta\in\tau_k.$$
Therefore, 
$$  \sum_{k=1}^\infty A_k^2\int_{\tau_k}|\varphi_1|^2|g|^2{\mathrm d}m
\leq \|g\|^2\sum_{k=1}^\infty C_{n_k+1}^2 A_k^2\int_{\tau_k}|\varphi_1|^2{\mathrm d}m.$$

In both cases, we obtain the sequences  $\{A_k\}_{k=1}^\infty$ and $\{\tau_k\}_{k=1}^\infty$ such that $A_k>0$ for all $k\geq 1$, 
$A_k\to\infty$, $m(\tau_k)>0$, $\tau_k\cap\tau_l=\emptyset$ for $k\neq l$, $k$, $l\geq 1$, \eqref{aavarphi1}  is fulfilled,
\begin{equation}\label{aakdeltak} \sup_{k\geq 1}A_k\mathop{\text{\rm ess\,sup}}_{\tau_k}|\varphi_0||\varphi_1|\leq 1, 
\end{equation}
and there exists $C>0$ such that 
\begin{equation}\label{aavarphi1g} \sum_{k=1}^\infty A_k^2\int_{\tau_k}|\varphi_1|^2|g|^2{\mathrm d}m\leq C\|g\|^2
\ \text{ for every }  g\in\mathcal K_{\mathbf v}.
\end{equation}

There exists an outer function $\varphi\in H^2$ such that $$|\varphi|=\begin{cases} |\varphi_1| & m\text{-a.e. on } 
\mathbb T\setminus\cup_{k=1}^\infty \tau_k, \\ 
A_k|\varphi_1|&  m\text{-a.e. on } \tau_k, \ \ k\geq 1.\end{cases}$$ 
It is easy to see that $\varphi$ satisfies to the conclusion of the lemma. 
Indeed, $\varphi\in H^2$ by \eqref{aavarphi1}. 
The inclusion $\varphi\varphi_0\in H^\infty$ follows from   the inclusion $\varphi_0\varphi_1\in H^\infty$ and \eqref{aakdeltak}. Since $A_k\to\infty$ and  $m(\tau_k)>0$ for all $k\geq 1$, we have 
$\varphi/\varphi_1\not\in H^\infty$. 
Let $g\in\mathcal K_{\mathbf v}$. Then
\begin{equation*}\begin{aligned}\int_{\mathbb T}|\varphi|^2|g|^2{\mathrm d}m &
 =\int_{\mathbb T\setminus\cup_{k=1}^\infty \tau_k}|\varphi|^2|g|^2{\mathrm d}m+
\sum_{k=1}^\infty \int_{\tau_k}|\varphi|^2|g|^2{\mathrm d}m\\&=
\int_{\mathbb T\setminus\cup_{k=1}^\infty \tau_k}|\varphi_1|^2|g|^2{\mathrm d}m+
\sum_{k=1}^\infty A_k^2 \int_{\tau_k}|\varphi_1|^2|g|^2{\mathrm d}m.\end{aligned}\end{equation*}
The first summand is finite by the assumption on $\varphi_1$, while the second summand is finite by \eqref{aavarphi1g}.
\end{proof}

\begin{theorem} \label{thmssmain} Suppose that ${\mathbf u}$, ${\mathbf v}$ and $\varphi_0$ satisfy the assumptions of Proposition 
\ref{propdeftt} and Lemma \ref{lemssmain}. Namely, ${\mathbf u}$, ${\mathbf v}$ are inner functions, $\varphi_0\in H^\infty$, 
$1/\varphi_0\not\in H^2$, \eqref{closvarphi0} is fulfilled and $m(E_{\mathbf v})=1$, where $E_{\mathbf v}$ is defined by 
\eqref{eev}. Define $T$ as in Proposition  \ref{propdeftt}. Then $T\sim S$ and $T$ does not have an isometric asymptote.
\end{theorem}

\begin{proof} The relation $T\sim S$ is proved in Proposition  \ref{propdeftt}. If $T$ has an isometric asymptote, then, by  Lemmas \ref{lemfinite} or \ref{lemshift}, it must be $S$. 
 Assume that $X_1$ is a canonical intertwining mapping for $T$ and $S$. By Proposition  \ref{propdeftt}, there exists $\psi_1\in H^\infty$ 
which satisfies \eqref{psiv},  $X_1$ acts according to \eqref{zzdef} with $\psi_1$, and $X_1Y=\psi_1(S)$.  Set $\varphi_1=\psi_1/\varphi_0$. Then $\varphi_1$ satisfies the assumption of Lemma \ref{lemssmain}
(with $\varphi_0$ and $\mathbf v$). 
Indeed, $\varphi_1\in H^2$, because  $\varphi_1\mathcal K_{\mathbf v}\subset H^2$ and  there exist  
$f\in\mathcal K_{\mathbf v}$ such that $f$, $1/f\in H^\infty$.

Let $\varphi$ be from the conclusion of Lemma \ref{lemssmain}. 
Set $\psi=\varphi\varphi_0$ and define $Z$ by \eqref{zzdef}. Then there exists $\psi_0\in H^\infty$ such that 
  $Z=\psi_0(S)X_1$. By \eqref{zzyy}, $$\psi(S)=ZY=\psi_0(S)X_1Y=\psi_0(S)\psi_1(S).$$
Therefore, $\psi=\psi_0\psi_1$. Consequently, $\varphi=\psi_0\varphi_1$, which  contradicts to the relation
 $\varphi/\varphi_1\not\in H^\infty$.
\end{proof}

For examples of ${\mathbf u}$, ${\mathbf v}$ and $\varphi_0$ satisfing the assumptions of Theorem \ref{thmssmain}
we refer to \cite{fhr}. One of examples are some concrete functions from {\cite[Theorem 6.14]{fhr}}. Equality  \eqref{closvarphi0} 
follows from (ii) in the proof of  {\cite[Theorem 6.14]{fhr}}. The relations $\varphi_0\in H^\infty$ and  
$1/\varphi_0\not\in H^2$ follow from the estimate 
$$|\varphi_0(\zeta)|\asymp \Bigl(1+\Bigl|\frac{1+\zeta}{1-\zeta}\Bigr|\Bigl)^{-1/2} \ \text{ for } 1\neq\zeta\in\mathbb T,$$
which follows from the estimate of the correspondent function in the upper half plane obtained in the proof of 
{\cite[Theorem 6.14]{fhr}}.

In the following example   {\cite[Section 7]{fhr}} is used. 

Let $\sigma$ be  a singular (with respect to $m$) positive Borel  measure on $\mathbb T$, and let $\sigma(\mathbb T)=1$. 
Then the function ${\mathbf u}$ defined by the formula
\begin{equation}
\label{yy0}
\frac{1}{1-{\mathbf u}(z)}=\int_{\mathbb T}\frac{1}{1-z\overline\zeta}\text{\rm d}\sigma(\zeta),  \ \ z\in\mathbb D,\end{equation}
is inner, and ${\mathbf u}(0)=0$. 
Conversely, for every inner function ${\mathbf u}$ such that ${\mathbf u}(0)=0$ there exists 
a singular positive Borel measure $\sigma_{\mathbf u}$ on $\mathbb T$ called the \emph{Clark measure} such that 
$\sigma_{\mathbf u}(\mathbb T)=1$ and \eqref{yy0} is fulfilled with $\sigma=\sigma_{\mathbf u}$.  
It is easy to see that ${\mathbf u}$ has analytic continuation on $\mathbb C\setminus\operatorname{supp}\sigma_{\mathbf u}$, 
where $\operatorname{supp}\sigma_{\mathbf u}$ is the \emph{closed} support of $\sigma_{\mathbf u}$. 

Every function $f\in\mathcal K_{\mathbf u}$ has nontangential boundary values $f(\zeta)$ for $\sigma_{\mathbf u}$-a.e. 
 $\zeta\in\mathbb T$. The transformation 
$V_{\mathbf u}\in\mathcal L(L^2(\mathbb T,\sigma_{\mathbf u}),\mathcal K_{\mathbf u} )$ defined by the formula  
$$(V_{\mathbf u}\gamma)(z)=(1-{\mathbf u}(z))\int_{\mathbb T}\frac{\gamma(\zeta)}{1-z\overline\zeta}\text{\rm d}\sigma_{\mathbf u}(\zeta),  
\ \ z\in\mathbb D,\ \ \gamma\in L^2(\mathbb T,\sigma_{\mathbf u}), $$
is unitary, and 
$$ (V_{\mathbf u}^{-1}f)(\zeta)=f(\zeta) \ \ \sigma_{\mathbf u}\text{-a.e. }\ \zeta\in\mathbb T, 
 \ \   f\in\mathcal K_{\mathbf u}.$$
For references, see \cite{clark} and \cite{polt},  also {\cite[Ch. 9 and Theorem 10.3.1]{cauchy}}, {\cite[Sec. 8]{model}} and references therein.  

\begin{lemma}\label{lemfhr}{\cite[Sec. 7]{fhr}} Suppose that ${\mathbf u}$ and ${\mathbf v}$ are two inner functions, 
 ${\mathbf u}(0)={\mathbf v}(0)=0$,  $\mathbf h\in L^\infty(\mathbb T,\sigma_{\mathbf v})$ is such that \begin{equation}\label{neq0}{\mathbf h}> 0 \ \ \sigma_{\mathbf v}\text{-a.e. on }\mathbb T
\end{equation}
 and \begin{equation}\label{uvh}\sigma_{\mathbf u}={\mathbf h}\sigma_{\mathbf v}. \end{equation}
Set \begin{equation}\label{uvphi}\varphi_0=\frac{1-{\mathbf v}}{1-{\mathbf u}}.\end{equation} 
Then \eqref{closvarphi0} is fulfilled, and $1/\varphi_0\in H^2$ if and only if $1/{\mathbf h}\in L^1(\mathbb T,\sigma_{\mathbf v})$.
\end{lemma}

\begin{proof} 
By {\cite[Corollary 7.6]{fhr}}, $\varphi_0$ is a multiplier from $\mathcal K_{\mathbf u}$ to $\mathcal K_{\mathbf v}$. 
Denote by $M$ the operator of multiplication by $\varphi_0$, then $M\in\mathcal L(\mathcal K_{\mathbf u},\mathcal K_{\mathbf v})$. Denote by $J$ the natural imbedding of $L^2(\mathbb T,\sigma_{\mathbf v})$ into $ L^2(\mathbb T,\sigma_{\mathbf u})$. 
By the proof of  {\cite[Theorem 7.3]{fhr}}, $M^*=V_{\mathbf u}J V_{\mathbf v}^{-1} $. 
By \eqref{neq0}, $\ker J=\{0\}$, which is equivalent to the equality $\ker M^*=\{0\}$, 
and   \eqref{closvarphi0} follows. 

It follows from \eqref{neq0} that  $\sigma_{\mathbf v}=(1/{\mathbf h})\sigma_{\mathbf u}$. By {\cite[Remark 7.7]{fhr}} 
(applied to ${\mathbf v}$, ${\mathbf u}$ instead of ${\mathbf u}$, ${\mathbf v}$), $1/\varphi_0\in H^2$ if and only if 
$1/{\mathbf h}\in L^2(\mathbb T,\sigma_{\mathbf u})$, which is equivalent to the inclusion 
 $1/{\mathbf h}\in L^1(\mathbb T,\sigma_{\mathbf v})$. 
\end{proof}

\begin{lemma}\label{lemlip} Suppose that ${\mathbf u}$ and ${\mathbf v}$ are two inner functions, 
 ${\mathbf u}(0)={\mathbf v}(0)=0$, $C>0$, and $\mathbf h\colon\overline{\mathbb D}\to[0,+\infty)$ is a function 
such that 
\begin{equation}\label{lip}|{\mathbf h}(z)-{\mathbf h}(\zeta)|\leq C|z-\zeta| \ \text{ for all }z,\  \zeta\in\overline{\mathbb D}. 
\end{equation}
Furthermore, suppose that \eqref{uvh} is fulfilled, and $\varphi_0$ is defined by \eqref{uvphi}. 
Then $\varphi_0\in H^\infty$. 
\end{lemma}

\begin{proof} By \eqref{yy0},
\begin{equation*}
\begin{aligned}\varphi_0(z)& =(1-{\mathbf v}(z))
\int_{\mathbb T}\frac{{\mathbf h}(\zeta)\text{\rm d}\sigma_{\mathbf v}(\zeta)}{1-z\overline\zeta}\\&
=(1-{\mathbf v}(z))\int_{\mathbb T}
\frac{{\mathbf h}(\zeta)-{\mathbf h}(z)}{\zeta-z}\zeta\text{\rm d}\sigma_{\mathbf v}(\zeta)
\\&\quad+{\mathbf h}(z)(1-{\mathbf v}(z))
\int_{\mathbb T}\frac{\text{\rm d}\sigma_{\mathbf v}(\zeta)}{1-z\overline\zeta},   
\ \   \ \ z\in\mathbb D.\end{aligned}\end{equation*}
The first summand is bounded due to \eqref{lip}, while the second summand is equal to ${\mathbf h}(z)$ by \eqref{yy0}. 
\end{proof}

Lemmas \ref{lemfhr} and \ref{lemlip} allow to construct ${\mathbf u}$, ${\mathbf v}$ and $\varphi_0$ satisfying 
the assumptions of Theorem \ref{thmssmain}. For example, take ${\mathbf h}_1$ satisfying \eqref{lip} such that 
$ {\mathbf h}_1(z)\neq 0$ for $z\in\mathbb T\setminus\{1\}$, and   ${\mathbf h}_1(1)=0$. 
Let $\{\zeta_n\}_{n=1}^\infty\subset\mathbb T\setminus\{1\}$ 
be such that $\zeta_n\neq\zeta_k$, if $n\neq k$, and $\zeta_n\to 1$. There exists a sequence $\{a_n\}_{n=1}^\infty$ such that 
$a_n>0$ for all $n\geq 1$, $\sum_{n=1}^\infty a_n=1$ and $$\sum_{n=1}^\infty \frac{a_n}{{\mathbf h}_1(\zeta_n)}= \infty.$$
Set $\sigma_{\mathbf v}=\sum_{n=1}^\infty a_n{\boldsymbol\delta}_{\zeta_n}$ and 
$$\sigma_{\mathbf u} = \frac{1}{\sum_{n=1}^\infty a_n{\mathbf h}_1(\zeta_n)} \sum_{n=1}^\infty a_n{\mathbf h}_1(\zeta_n){\boldsymbol\delta}_{\zeta_n},$$
where ${\boldsymbol\delta}_{\zeta}$ is the Dirac measure at $\zeta$.  Then ${\mathbf u}$, ${\mathbf v}$ and $\varphi_0$ defined by \eqref{yy0} and \eqref{uvphi} satisfy 
the assumptions of Theorem \ref{thmssmain}.

\section{On isometric asymptote of contractions quasisimilar to the unilateral shift of infinite multiplicity. }

In this section, an example of an operator $T$ (on a Hilbert space) will be constructed such that $T$ is similar to a contraction, $T\sim S_\infty$, and $S_\infty$ is not an isometric asymptote of $T$. 

Let $T\in\mathcal L(\mathcal H)$ have a unitary asymptote $(X,U)$, where $U\in\mathcal L(\mathcal K)$. 
For $ \mathcal E\subset \mathcal K$ set 
$$X^{-1}\mathcal E =\{x\in\mathcal H\ :\ Xx\in\mathcal E \}.$$
Then $X^{-1}\mathcal E\in\operatorname{Hlat}T$ for every $\mathcal E\in\operatorname{Hlat}U$. 
It is well known that  $\operatorname{Hlat}U\neq\{\{0\},\mathcal K\}$, if $U$ is not the multiplication by a unimodular constant on 
 $\mathcal K$. But it is possible that 
\begin{equation}\label{defquasi} X^{-1}\mathcal E =\{0\} 
\ \text{ for every } \mathcal E\in\operatorname{Hlat}U \text{ such that }  \mathcal E\neq\mathcal K.
\end{equation}
Such an operator $T$ is called \emph{quasianalytic}. For references, see \cite{ker16},  \cite{kerntuple}, \cite{kersz} and references therein.

\begin{proposition}\label{propjjttrr} 
Suppose that $\mathcal H_0$, $\mathcal K$, $\mathcal G$ are Hilbert spaces, 
$R\in\mathcal L(\mathcal K)$, $U\in\mathcal L(\mathcal G)$ is unitary,  $(Y, U)$ is a unitary asymptote of $R$, 
and there exists $\mathcal E\in \operatorname{Hlat}U$ such that 
$$\{0\}\neq\mathcal M:=Y^{-1}\mathcal E=\{y\in\mathcal K\ :\ Yy\in \mathcal E\}\neq\mathcal K.$$ 
Set $$\mathcal K_0=\mathcal K\ominus \mathcal M \ \text{ and }\  R_0=P_{\mathcal K_0}R|_{\mathcal K_0}.$$
Suppose that 
$T_0\in\mathcal L(\mathcal H_0)$ is a quasianalytic operator,  and 
 $J\in\mathcal L(\mathcal H_0,\mathcal K)$ is a quasiaffinity such that $JT_0=RJ$ and 
$(YJ, U)$ is a unitary asymptote of $T_0$.
Then $P_{\mathcal K_0}J$ is a quasiaffinity and  $P_{\mathcal K_0}JT_0=R_0P_{\mathcal K_0}J$. 
\end{proposition}

\begin{proof} The density of the range and the  intertwining property of $P_{\mathcal K_0}J$ follow from 
the density of the  range and the  intertwining property of $J$.
Since $$\ker P_{\mathcal K_0}J=\{x\in\mathcal H_0\ :\ Jx\in\mathcal M\}=
\{x\in\mathcal H_0\ :\ YJx\in\mathcal E\}$$ 
and $\mathcal E\neq \mathcal G$ (because $\mathcal M\neq\mathcal K$), the quasianalyticity of $T_0$ implies that  $\ker P_{\mathcal K_0}J=\{0\}.$ 
\end{proof}

\begin{lemma}\label{lemnnee} Suppose that $U$ is a unitary operator, $V$ is an isometry,  
$\mathcal E\in \operatorname{Lat}U$ is such that $U|_{\mathcal E}$ is a reductive unitary operator, and  
$\mathcal N\in \operatorname{Lat}(V\oplus U)$ is such that $\mathcal N\cap\mathcal E\neq\{0\}$. 
Then $ (V\oplus U)|_{\mathcal N}$ is not a unilateral shift  (of finite or infinite multiplicity).
\end{lemma}

\begin{proof} Clearly,  $(V\oplus U)|_{\mathcal N\cap\mathcal E}=U|_{\mathcal N\cap\mathcal E}$.  
Since $U|_{\mathcal E}$ is  reductive,  $U|_{\mathcal N\cap\mathcal E}$  is a non-zero unitary. Therefore, 
$ (V\oplus U)|_{\mathcal N}$ is not a unilateral shift. 
\end{proof}

The following theorem is a particular case of {\cite[Theorem VI.2.3]{nfbk}}.

\begin{theoremcite}[{\cite[Theorem VI.2.3]{nfbk}}]\label{thmnfbk} Let $\mathcal K_0$ be a Hilbert space, and let 
$R_0\in\mathcal L(\mathcal K_0)$ be a contraction of class $C_{\cdot 0}$. Then there exists 
$K\in\mathcal L(\mathcal K_0,H^2_\infty)$ 
such that 
\begin{equation}\label{congss}V:=\begin{pmatrix}S_\infty & K\\ \mathbb O& R_0\end{pmatrix}\cong S_\infty.\end{equation}
\end{theoremcite}

\begin{theorem}\label{thm5}
Suppose that $R$, $Y$, $T_0$, $J$ satisfy Proposition \ref{propjjttrr},  $R$ and $T_0$ are contractions  of class 
 $C_{10}$, 
and the pairs $R$, $Y$ and $T_0$, $YJ$ satisfy \eqref{lim}. 
Let  $R_0$ and $\mathcal K_0$ be from Proposition \ref{propjjttrr}. Let $0<c<1$,  and let   $X_0=cP_{\mathcal K_0}J$. 
Set $$T=\begin{pmatrix}S_\infty & KX_0\\ \mathbb O& T_0\end{pmatrix},$$
where $K$ is from Theorem \ref{thmnfbk} applied to $R_0$. Then $T$ is similar to a contraction, $T\sim S_\infty$, and 
 an isometric asymptote of $T$ contains  a (non-trivial) unitary summand. \end{theorem}

\begin{proof}
Denote by  $K_n\in\mathcal L(\mathcal K_0,H^2_\infty)$, $n\geq 1$, the transformations such that 
$$V^n=\begin{pmatrix}S_\infty & K\\ \mathbb O& R_0\end{pmatrix}^n =\begin{pmatrix}S_\infty^n & K_n\\ \mathbb O& R_0^n\end{pmatrix}. $$
By Proposition \ref{propjjttrr},  $X_0$  is a quasiaffinity, and  $X_0T_0=R_0X_0$. 
It follows from the definitions of $T$ and $V$ that 
 \begin{equation}\label{ttxx0}
\begin{pmatrix}I_{H^2_\infty} &\mathbb O \\  \mathbb O& X_0\end{pmatrix}T=V
\begin{pmatrix}I_{H^2_\infty} &\mathbb O \\  \mathbb O& X_0\end{pmatrix}.\end{equation}
It is easy to see from \eqref{ttxx0} that  for every $n\geq 1$ 
$$T^n=\begin{pmatrix}S_\infty^n& K_nX_0\\ \mathbb O& T_0^n\end{pmatrix}.$$
Therefore, $T$ is power bounded. By {\cite[Corollary 4.2]{Cas}}, $T$ is similar to a contraction. 
 Since $X_0$ is a quasiaffinity, \eqref{ttxx0} and \eqref{congss} imply the relation $T\prec S_\infty$. 
By {\cite[Corollary 3]{tak}}, $T\sim S_\infty$. It remains  to prove that an isometric asymptote of $T$ contains  a (non-trivial) unitary summand. 

Set  $\mathcal G_0=\mathcal G\ominus\mathcal E$. Let $h_1$, $h_2\in H^2_\infty$, let $x_1$, $x_2\in\mathcal H_0$, and let $n\geq 1$. A straightforward 
calculation shows that
\begin{equation*}\begin{aligned} (T^n(h_1\oplus x_1), T^n(h_2\oplus x_2))&=(h_1,h_2)+(X_0x_1,X_0x_2)+(T_0^nx_1,T_0^nx_2)\\&\quad  -(R_0^nX_0x_1,R_0^nX_0x_2).\end{aligned}\end{equation*}
By \eqref{lim} applied to $T_0$ and $YJ$, 
$$\lim_n(T_0^nx_1,T_0^nx_2)=(YJx_1,YJx_2).$$
By {\cite[Theorem 3(c)]{ker89}}, $(P_{\mathcal G_0}Y|_{\mathcal K_0}, U|_{\mathcal G_0})$ is a unitary asymptote of $R_0$ which satisfies \eqref{lim}. 
Therefore, $$\lim_n(R_0^nX_0x_1,R_0^nX_0x_2)=(P_{\mathcal G_0}YX_0x_1,P_{\mathcal G_0}YX_0x_2).$$
Taking into account that $P_{\mathcal G_0}YX_0=cP_{\mathcal G_0}YJ$ and setting 
$$A=(1-c^2)^{1/2}P_{\mathcal G_0}\oplus P_{\mathcal E},$$
we conclude that 
$$ \lim_n(T^n(h_1\oplus x_1), T^n(h_2\oplus x_2))=(h_1,h_2)+(X_0x_1,X_0x_2)+(AYJx_1,AYJx_2).$$
Define $X_1\in\mathcal L(H^2_\infty\oplus\mathcal H_0, H^2_\infty\oplus\mathcal K_0\oplus\mathcal G)$ by the formula 
$$X_1(h\oplus x)=h\oplus X_0x\oplus AYJx, \ \ \ \ h\in H^2_\infty, \ \  x\in\mathcal H_0.$$
Then $T$, $X_1$ satisfy \eqref{lim}. 
Since  $\mathcal E\in \operatorname{Hlat}U$, we have $AU=UA$. Therefore, $X_1T=(V\oplus U)X_1$. 
Set $\mathcal N=\operatorname{clos} X_1(H^2_\infty\oplus\mathcal H_0)$.
Consider $X_1$ as  a transformation acting into $\mathcal N$ and denote it by $X$. Then 
 $X\in\mathcal L(H^2_\infty\oplus\mathcal H_0, \mathcal N)$ and  $(X, (V\oplus U)|_{\mathcal N})$  
is an isometric asymptote of $T$ by \eqref{lim} and \cite{ker89}. 

Since $R$ is of class $C_{1\cdot}$, we have $\ker Y=\{0\}$. Consequently, $Y\mathcal M\neq\{0\}$. 
Let $y\in\mathcal M$. There exists  a sequence $\{x_n\}_n\subset\mathcal H_0$  such that $Jx_n\to y$. Then 
$X_0x_n=cP_{\mathcal K_0}Jx_n\to 0$ and $YJx_n\to Yy$.  Since $Yy\in\mathcal E$, we have 
$$AYy =((1-c^2)^{1/2}P_{\mathcal G_0}\oplus P_{\mathcal E})Yy=Yy.$$ We conclude that 
$$X(0\oplus x_n)\to 0\oplus 0\oplus AYy = 0\oplus 0\oplus Yy.$$
Thus, $\mathcal N\cap\mathcal E\supset\operatorname{clos}Y\mathcal M\neq\{0\}$.
By Lemma \ref{lemnnee}, $(V\oplus U)|_{\mathcal N}$ is not a unilateral shift. 
\end{proof}

In Example \ref{exaweight} contractions $T_0$ and $R$ taking from \cite{est} and \cite{kersz}   are given which satisfy the assumptions of Theorem \ref{thm5}. 

\begin{example}\label{exaweight}
The definition of bilateral weighted shift $\mathcal S_\omega$  is recalled in Introduction. 
 Notation from Introduction is used.  
Let $\omega_0$ and $\omega$ be two unbounded nonincreasing weights such that $\omega(n)=\omega_0(n)=1$ for all $n\geq 0$ and 
\begin{equation}\label{omega}\omega_0(-n)\geq\omega(-n) \text{   for all } n\geq1.\end{equation} 
Then the natural imbedding $J$ of $L^2_{\omega_0}(\mathbb T)$ into $L^2_\omega(\mathbb T)$ is a quasiaffinity and 
$Y_{\omega_0}=Y_\omega J$. 

A weight $\omega$ is called \emph{quasianalytic} if 
$$ \sum_{n=1}^\infty\frac{\log\omega(-n)}{n^2}=\infty$$
 and is called \emph{regular} if the sum above is finite. 
Moreover,  a  weight $\omega$ is quasianalytic if and only if  $\mathcal S_\omega$ is a quasianalytic contraction 
in the sense of definition \eqref{defquasi}, see {\cite[Propositon 31]{kersz}}, see also {\cite[Sec. 4]{est}}. 
Thus, if  $\omega_0$ and $\omega$ are two  weights which satisfy \eqref{omega} and such that $\omega_0$ is quasianalytic while $\omega$ 
is regular, then $T_0=\mathcal S_{\omega_0}$ and $R=\mathcal S_\omega$ satisfy  the assumptions of Theorem \ref{thm5}. 
For concrete weights, take  $0<\alpha<1$ and set $\omega_0(-n)=e^n$ and 
$\omega(-n)=e^{n^\alpha}$ for $n\geq 1$ (and $\omega_0(n)=\omega(n)=1$ for $n\geq 0$). 
\end{example}

\end{document}